\documentclass[11pt,reqno]{amsart}
\usepackage{amsthm,amssymb,amsfonts}
\usepackage[nosumlimits]{mathtools}
\usepackage{mathrsfs}
\usepackage{graphicx}
\usepackage{subcaption}
\usepackage{algorithm}
\usepackage{algpseudocode}
\usepackage{tikz-cd}
\usepackage{enumerate}
\usepackage{quiver}
\usepackage{enumitem}
\usepackage{accents}
\usepackage[margin=1in]{geometry}

\usepackage{hyperref,xcolor}
\hypersetup{
    colorlinks,
    linkcolor={red!50!black},
    citecolor={blue!50!black},
    urlcolor={blue!80!black}
}
\usepackage{adjustbox}

\theoremstyle{plain}
\newtheorem{theorem}{Theorem}[section]
\newtheorem{proposition}[theorem]{Proposition}
\newtheorem{lemma}[theorem]{Lemma}
\newtheorem{corollary}[theorem]{Corollary}
\newtheorem{remark}[theorem]{Remark}
\theoremstyle{definition}
\newtheorem{definition}[theorem]{Definition}

\newtheorem{conjecture}[theorem]{Conjecture}

\DeclareMathOperator{\GL}{GL}

\DeclareMathOperator{\AR}{AR}
\DeclareMathOperator{\GR}{GR}

\DeclareMathOperator{\Q}{Q}
\DeclareMathOperator{\codim}{codim}
\DeclareMathOperator{\spa}{span}
\DeclareMathOperator{\Hom}{Hom}

\DeclareMathOperator{\PR}{PR}
\DeclareMathOperator{\SR}{SR}
\DeclareMathOperator{\Gr}{Gr}
\DeclareMathOperator{\height}{ht}
\DeclareMathOperator{\Ass}{Ass}
\DeclareMathOperator{\rank}{rank}
\DeclareMathOperator{\Min}{Min}
\DeclareMathOperator{\str}{str}
\DeclareMathOperator{\Spec}{Spec}
\DeclareMathOperator{\ch}{char}
\DeclareMathOperator{\bias}{bias}
\DeclareMathOperator{\Prob}{Prob}

\begin{document}
\title{Bounds for geometric rank in terms of subrank}
\date{}
\author[Q-Y.~Chen]{Qiyuan~Chen}
\address{State Key Laboratory of Mathematical Sciences, Academy of Mathematics and Systems Science, Chinese Academy of Sciences, Beijing 100190, China}
\email{chenqiyuan@amss.ac.cn}
\author[K.~Ye]{Ke Ye}
\address{State Key Laboratory of Mathematical Sciences, Academy of Mathematics and Systems Science, Chinese Academy of Sciences, Beijing 100190, China}
\email{keyk@amss.ac.cn}
\begin{abstract}
For tensors of fixed order,  we establish three types of upper bounds for the geometric rank in terms of the subrank.  Firstly,  we prove that,  under a mild condition on the characteristic of the base field,  the geometric rank of a tensor is bounded by a function in its subrank in some field extension of bounded degree.  Secondly,  we show that,  over any algebraically closed field,  the geometric rank of a tensor is bounded by a function in its subrank.  Lastly,  we prove that,  for any order three tensor over an arbitrary field,  its geometric rank is bounded by a quadratic polynomial in  its subrank.  Our results have several immediate but interesting implications: \emph{(1)} We answer an open question posed by Kopparty,  Moshkovitz and Zuiddam concerning the relation between the subrank and the geometric rank;  \emph{(2)} For order three tensors,  we generalize the Biaggi-Chang-Draisma-Rupniewski (resp.  Derksen-Makam-Zuiddam) theorem on the growth rate of the border subrank (resp.  subrank),  in an optimal way;  \emph{(3)} For order three tensors,  we generalize the Biaggi-Draisma-Eggleston theorem on the stability of the subrank,  from the real field to an arbitrary field; \emph{(4)} We confirm the open problem raised by Derksen,  Makam and Zuiddam on the maximality of the gap between the subrank of the direct sum and the sum of subranks; \emph{(5)} We derive,  for the first time,  a de-bordering result for the border subrank and upper bounds for the partition rank and analytic rank in terms of the subrank; \emph{(6)} We reprove a gap result for the subrank.
\end{abstract}
\maketitle


\section{Introduction and main results}
\subsection{Introduction}
As a generalization of matrices,  tensors are widespread across several fields of mathematics, often known by different names.  They appear as multilinear polynomials in algebraic geometry \cite{adiprasito2021schmidt,LZ24,bik2025strength}; as multilinear maps in complexity theory \cite{Landsberg12,kopparty2020geometric,CM22}; as multilinear functions \cite{GW11, CLP17,Gowers21} in combinatorics; and as multiarrays in numerical analysis \cite{LP10,Derksen16,HY23}.  In these respective fields,  various tensor ranks have been introduced to facilitate a quantitative investigation of the multifaceted nature of tensors.  For instance,  the \emph{subrank} \cite{Strassen87} is proposed to measure the complexity of tensors; the \emph{analytic rank} \cite{GW11} quantifies the randomness of tensors; while the \emph{slice rank} \cite{Terrence16} and the \emph{partition rank} \cite{Naslund20} are used to characterize the structure of tensors; the \emph{$G$-stable rank} \cite{Derksen22} is defined to study tensors by the geometric invariant theory; and the \emph{geometric rank} \cite{kopparty2020geometric} is employed to characterize the geometry of tensors.  

Since the pioneering work \cite{GT09},  it has been realized that some of the aforementioned tensor ranks are deeply related to each other,  indicating the interplay among seemingly different structures characterized by these ranks.  For example,  multiple authors \cite{GW11,lovett2018analytic,adiprasito2021schmidt} have conjectured that the partition rank is linearly equivalent to the analytic rank,  which has been confirmed over large fields \cite{cohen2023partition}.  The partition rank is also conjectured to be linearly equivalent to the geometric rank \cite{Schmidt85,adiprasito2021schmidt}.  Moreover,  the linear equivalence between the analytic rank and the geometric rank has been independently established in \cite{chen2024stability,baily2024strength},  and it has been proved that the slice rank is linearly equivalent to the $G$-stable rank \cite{Derksen22}. Notably,  the linear equivalence among these ranks is equivalent to the stability and the asymptotic additivity of the partition rank.  We refer interested readers to \cite{chen2024stability} and references therein for more details.   

Among all the tensor ranks discussed above,  the subrank is especially distinctive,  for the following two reasons: \emph{(1)} The subrank lies at the intersection of the complexity theory \cite{Strassen88,Strassen91,AFL15,kopparty2020geometric,CVZ21} and combinatorics \cite{Terrence16,CLP17,NS17,FGL21},  playing crucial roles in both fields.  Thus,  the subrank may potentially serve as a bridge between the complexity theory and combinatorics.  For comparison,  other ranks do not exhibit such a feature; \emph{(2)} Over an algebraically closed field,  the subrank of a generic $n$-dimensional,  order-$d$ tensor is $\Theta(n^{1/{(d-1)}})$ \cite{derksen2024subrank},  whereas the generic value of all other ranks is $\Theta(n)$.  This implies a large separation between the subrank and other ranks for a generic tensor,  making it challenging to even conjecture a precise relation between the subrank and other ranks.  By contrast,  it is worth noting that the linear equivalence among other ranks has been extensively studied in recent years \cite{GW11,Milicevi19,adiprasito2021schmidt,Derksen22, cohen2023partition, GD24,chen2024stability}.

Suppose $\mathbb{K}$ is a field.   Given a tensor $T \in \mathbb{K}^{n_1} \otimes \cdots \otimes \mathbb{K}^{n_d}$,  we denote by $\GR(T)$ (resp.  $\Q(T)$) the geometric rank (resp.  subrank) of $T$ (cf.  Definition~\ref{def:rank}).  It is observed in \cite{kopparty2020geometric} that $\Q(T) \le \GR(T)$.  However,  in the existing literature,  there is no known upper bound for $\GR(T)$ (or other ranks) in terms of $\Q(T)$.  In fact,  the existence of such an upper bound is posed as an open problem in \cite[Section~9]{kopparty2020geometric}.  The primary goal of this paper is to establish upper bounds for the geometric rank in terms of the subrank.  As applications,  we answer two open questions raised respectively in \cite[Section~9]{kopparty2020geometric} and \cite[Section~6]{derksen2024subrank}.  We also generalize existing results \cite[Theorem~1.2]{derksen2024subrank},  \cite[Theorem~2]{biaggi2025border} and \cite[Theorem~1.5]{biaggi2025real},  to a large extent.  Moreover,  we obtain three new results on the subrank and provide a simple proof for the gap theorem of the subrank.

\subsection{Main results}
The main contributions of this paper consist of the three theorems stated below,  each addressing a distinct but equally important scenario.  Given functions $f,  g,  h$ on $\mathbb{N}$,  we write $f = O(n)$,  $g = \Theta(n)$ and $f  \asymp h$ if there are positive constants $c,c_1,  c_2,c_3,  c_4$ such that  
\[
f(n) \le c n,\quad c_1 n  \le g(n) \le c_2 n,\quad c_3 f \le h \le c_4 f
\]
for sufficiently large $n \in \mathbb{N}$.  For a tensor $T \in \mathbb{K}^{n_1} \otimes \cdots \otimes \mathbb{K}^{n_d}$,  definitions of the subrank $\Q(T)$,  the border subrank $\underline{\Q}(T)$ and the geometric rank $\GR(T)$ can be found in Definition~\ref{def:rank}. Moreover,  for each field extension $\mathbb{F}$ of $\mathbb{K}$,  we may regard $T$ as a tensor in $\mathbb{F}^{n_1} \otimes_{\mathbb{F}} \cdots  \otimes_{\mathbb{F}} \mathbb{F}^{n_d}$ as follows.  Since $T$ is a tensor over $\mathbb{K}$,  we may write $T$ as a multiarray $T = (T_{i_1,\dots, i_d})$ where $T_{i_1,\dots,  i_d} \in \mathbb{K}$ and $i_1 \in [n_1],  \dots,  i_d \in [n_d]$.  By the inclusion $\mathbb{K} \subseteq \mathbb{F}$,  $T$ can also be viewed as a multiarrary over $\mathbb{F}$. 

\begin{theorem}[$\GR(T)$ vs.  $\Q(T)$ over large fields]\label{thm:general}
There exist functions $A$ and $B$ on $\mathbb{N} \times \mathbb{N}$ with the following property.  Suppose $\mathbb{K}$ is a field such that either $\ch(\mathbb{K})=0$ or $\ch(\mathbb{K})>d$.  For any tensor $T \in \mathbb{K}^{n_1} \otimes \cdots \otimes \mathbb{K}^{n_d}$,  there is a field extension $\mathbb{L}/\mathbb{K}$ such that  $[\mathbb{L} : \mathbb{K}] \le B(d, \GR(T))$ and $\GR(T) \le A(d,  \Q_{\mathbb{L}} (T))$.  Here $\Q_{\mathbb{L}} (T)$ denotes the subrank of $T$ as an element in $\mathbb{L}^{n_1} \otimes_{\mathbb{L}} \cdots \otimes_{\mathbb{L}} \mathbb{L}^{n_d}$.
\end{theorem}
When $\mathbb{K}$ is algebraically closed field,  we can remove the condition on $\ch(\mathbb{K})$ in Theorem~\ref{thm:general}.
\begin{theorem}[$\GR(T)$ vs.  $\Q(T)$ over algebraically closed fields]\label{thm:general1}
There exists a function $C$ on $\mathbb{N} \times \mathbb{N}$ such that for any algebraically closed field $\mathbb{K}$ and any tensor $T \in \mathbb{K}^{n_1} \otimes \cdots \otimes \mathbb{K}^{n_d}$,  we have $\GR(T) \le C(d,  \Q(T))$.
\end{theorem}
\begin{remark}
Theorem~\ref{thm:general1} addresses the open problem in \cite[Section~9]{kopparty2020geometric} regarding the relation between the (border) subrank and the geometric rank over an algebraically closed field.
\end{remark}

For $d = 3$,  Theorem~\ref{thm:general} can be strengthened to a rather large extend.  Indeed,  the condition on $\ch(\mathbb{K})$ can be removed,  $A(3,s)$ can be taken as a quadratic polynomial,  and $\mathbb{L}$ can be chosen to be $\mathbb{K}$.
\begin{theorem}[$\GR(T)$ vs.  $\Q(T)$ for order three tensors] \label{3-GR<Q infinite}
There are constants $c_1$ and $c_2$ such that for any field $\mathbb{K}$ and any tensor $T \in \mathbb{K}^{n_1} \otimes \mathbb{K}^{n_2} \otimes \mathbb{K}^{n_3}$,  we have 
\[
\GR(T) \le \begin{cases}
2 \Q(T)^2 + 3 \Q(T), \quad &\text{if $\mathbb{K}$ is an infinite field} \\
\frac{2c_2}{c_1} \Q(T)^2 + \left( \frac{2c_2}{c_1} + c_2 \right) \Q(T) +  \left( 2c_2 - \frac{c_2}{c_1} \right), \quad &\text{otherwise} \\
\end{cases}
\]
In particular,  we have $\GR(T) = O(\Q(T)^2) = O(\underline{\Q}(T)^2)$.
\end{theorem}
\begin{remark}
The quadratic upper bound of $\GR(T)$ in Theorem~\ref{3-GR<Q infinite} is optimal.  Indeed,  when $\mathbb{K}$ is algebraically closed,  we have $\Q(T) = \Theta(n^{1/2})$ \cite[Theorem~1.2]{derksen2024subrank} and $\GR(T) = n$ \cite[Lemma~5.3]{kopparty2020geometric} for a generic $T \in \mathbb{K}^{n} \otimes \mathbb{K}^n \otimes \mathbb{K}^n$.  Furthermore,  by \cite[Theorem~2]{biaggi2025border} or \cite[Corollary~2.17]{PSS25},  it generically holds that $\underline{\Q}(T) = \Theta(n^{1/2})$.  Consequently,  for a generic $T$,  we have  
\begin{equation}\label{eq:QvsGR}
\GR(T)  = \Theta ( \Q(T)^2 ),\quad  \GR(T) = \Theta (  \underline{\Q}(T)^2 ).
\end{equation} 
However,  we notice that it is impossible for \eqref{eq:QvsGR} to hold for all $T \in \mathbb{K}^{n} \otimes \mathbb{K}^n \otimes \mathbb{K}^n$.  For instance,  considering the identity tensor $I_n \in \mathbb{K}^{n} \otimes \mathbb{K}^n \otimes \mathbb{K}^n$,  we have $\GR(I_n) =  \Q(I_n)  =  \underline{\Q}(I_n) = n$.  Thus,  Theorem~\ref{3-GR<Q infinite} serves as an optimal generalization of \cite[Theorem~1.2]{derksen2024subrank} and \cite[Theorem~2]{biaggi2025border} to arbitrary order three tensors over any field. 
\end{remark}

\subsection{Applications}
Theorems~\ref{thm:general1} and \ref{3-GR<Q infinite} have several immediate applications,  each of which concerns a quantitative property of the subrank.  As the first application,  we show that the growth rate of the subrank of order three tensors is at most quadratic under field extensions.  We then prove a direct sum estimate and a de-bordering result for the subrank.  Next,  we respectively establish new upper bounds of the partition rank and the analytic rank in terms of the subrank.  In the last application,  we recover a gap theorem of the subrank.  Readers are referred to Definition~\ref{def:rank} for definitions of the partition rank,  the slice rank and the analytic rank.

\subsubsection{Stability of Subrank} 
For any $T \in \mathbb{K}^{n_1} \otimes \mathbb{K}^{n_2} \otimes \mathbb{K}^{n_3}$,  we observe that $\GR_{\overline{\mathbb{K}}} (T) = \GR(T)$.  Here $\GR_{\overline{\mathbb{K}}} (T)$ is the geometric rank of $T$ as a tensor in $\overline{\mathbb{K}}^{n_1} \otimes_{\overline{\mathbb{K}}} \overline{\mathbb{K}}^{n_2} \otimes_{\overline{\mathbb{K}}} \overline{\mathbb{K}}^{n_3}$.  Thus,  Theorem~\ref{3-GR<Q infinite} together with \cite[Theorem~5]{kopparty2020geometric} implies the following. 
\begin{corollary}[Stability of $\Q(T)$]\label{cor:Qstability}
For any field $\mathbb{K}$ and any tensor $T \in \mathbb{K}^{n_1} \otimes \mathbb{K}^{n_2} \otimes \mathbb{K}^{n_3}$,  we have $\Q_{\overline{\mathbb{K}}}( T )= O(\Q(T)^{2})$.
\end{corollary} 
\begin{remark}
In fact,  for $\mathbb{K} = \mathbb{R}$,  the proof of Corollary~\ref{cor:Qstability} yields $Q_{\mathbb{C}} (T) \le 2 \Q(T)^2 + 3 \Q(T)$.  On the other hand,  as shown in \cite[Theorem~1.5]{biaggi2025real},  a tighter bound $Q_{\mathbb{C}} (T) \le \Q(T)^2$ holds.   However,  when considering only the order of the upper bound,  Corollary~\ref{cor:Qstability} generalizes \cite[Theorem~1.5]{biaggi2025real} to arbitrary fields.   
\end{remark}


\subsubsection{Subrank of direct sum}
Given two order-$d$ tensors $S$ and $T$,  we have 
\[
\Q(S\oplus T) \le \GR(S\oplus T)= \GR(S)+\GR(T).
\] 
This together with Theorems~\ref{thm:general1} and \ref{3-GR<Q infinite} provides an upper bound for $\Q(S\oplus T)$.
\begin{corollary}[Subrank of direct sum]\label{direct sum} 
There is a function $C$ on $\mathbb{N}\times\mathbb{N}$ such that for any algebraically closed field $\mathbb{K}$ and tensors $T\in\mathbb{K}^{n_{1}}\otimes\cdots\otimes\mathbb{K}^{n_{d}}$ and $S\in \mathbb{K}^{m_{1}}\otimes\cdots\otimes\mathbb{K}^{m_{d}}$,  we have $\Q(S\oplus T)\le C(d, \Q(S))+C(d,\Q(T))$. Moreover,  we may take $C(3,r) = c r^2$ for some constant $c$. 
\end{corollary}
\begin{remark}
It is obvious that $\Q(S)+ \Q(T)\le \Q(S\oplus T)$. Only very recently,  the inequality was shown to be strict \cite[Theorem~5.2]{derksen2024subrank}, proving that the subrank is not additive under the direct sum.  Moreover,  the pair $(S,T)$ of order three tensors constructed in \cite[Theorem~5.1]{derksen2024subrank} satisfies 
\begin{equation}\label{eq:additive}
\Q(S \oplus T) \gtrsim (\Q(S) + \Q(T))^2.
\end{equation}
Based on this observation,  an open problem is proposed in  \cite[Section~6]{derksen2024subrank} asking whether the quadratic gap \eqref{eq:additive} between $\Q(S \oplus T)$ and $\Q(S) + \Q(T)$ is the largest possible.  Corollary~\ref{direct sum} confirms the optimality of this gap.   
\end{remark}

\subsubsection{De-bordering subrank}
We recall from \cite[Theorem~5]{kopparty2020geometric} that $\Q(T) \le \underline{\Q}(T) \le \GR(T)$ for any tensor $T$ over an algebraically closed field.  Thus,  Theorems~\ref{thm:general1} and \ref{3-GR<Q infinite} lead to the following bounds of $\underline{\Q}(T)$ in terms of $\Q(T)$.
\begin{corollary}[De-bordering of $\underline{\Q}(T)$]\label{cor:deborder}
There is a function $C$ on $\mathbb{N} \times \mathbb{N}$ such that for any algebraically closed field $\mathbb{K}$ and any tensor $T\in\mathbb{K}^{n_1}  \otimes \cdots \otimes \mathbb{K}^{n_d}$,  we have 
\[
\Q(T)\le \underline{\Q}(T) \le C(d,Q(T)).
\]
Moreover,  we may take $C(3,  r) = c r^2$ for some constant $c$. 
\end{corollary}
In algebraic complexity theory,  bounding a border complexity in terms of a non-border complexity is called a \emph{de-bordering} \cite{LL89,Kumar20, BDI21}.  Examples include the de-bordering of the Waring rank \cite{DGIJL24,PFCGV26} and de-bordering of the strength and the partition rank \cite{BDE19, bik2023uniformity,bik2025strength}.  In this context, Corollary~\ref{cor:deborder} can be understood as the de-bordering of the subrank.

\subsubsection{Subrank vs.  partition rank } 
By \cite{cohen2023partition},  we have $\GR(T) \asymp \PR(T)$ over any algebraically closed field. As a consequence,  we derive the following corollary of Theorems~\ref{thm:general1} and \ref{3-GR<Q infinite}.
\begin{corollary}[$\PR(T)$ vs.  $\Q(T)$]\label{cor:QvsSR}
There is a constant $c$ and a function $C_1$ on $\mathbb{N} \times \mathbb{N}$ such that for any field
 $\mathbb{K}$ and tensor $T \in \mathbb{K}^{n_1} \otimes \cdots \otimes \mathbb{K}^{n_d}$,  we have 
 \[
 \PR(T) 
\le \begin{cases} 
C_1(d,  \Q(T)),\quad &\text{if $\mathbb{K}$ is algebraically closed} \\ 
c \Q(T)^2,\quad &\text{if $d = 3$}
\end{cases}
 \]
In particular,  for any $T \in \mathbb{K}^{n_1} \otimes \mathbb{K}^{n_2} \otimes \mathbb{K}^{n_3}$,  we have $\SR(T) = O(\Q(T)^2 )$ and 
\begin{equation}\label{cor:QvsSR:eq1}
\liminf_{k \to \infty} \left[ \SR (T^{\boxtimes k})  \right]^{1/k} \le \limsup_{k \to \infty} \left[ \SR (T^{\boxtimes k})  \right]^{1/k}  \le \undertilde{\Q} (T) ^2 .
\end{equation}
where $\undertilde{\Q} (T) \coloneqq \lim_{k \to \infty} \left( \Q (T^{\boxtimes k})  \right)^{1/k}$ and $T^{\boxtimes k}$ is the tensor in $\mathbb{K}^{n_1^k} \otimes \cdots \otimes \mathbb{K}^{n_d^k}$ obtained by the Kronecker product of $k$ copies of $T$.
\end{corollary}
Given $T \in \mathbb{K}^{n_1} \otimes \mathbb{K}^{n_2} \otimes \mathbb{K}^{n_3}$,  the improved estimate $\limsup_{k \to \infty} \left[ \SR (T^{\boxtimes k})  \right]^{1/k} \le \undertilde{\Q}(T)^{3/2}$ is established in \cite[Theorem~12]{briet2023discreteness} by observing $\underline{\Q}(T^{\boxtimes 3}) = \Omega (\SR(T)^2) $.   However,  the proof of \cite[Theorem~12]{briet2023discreteness} is insufficient to conclude that $\SR(T) = O(\Q(T)^{3/2})$.  Actually,  this is impossible since by \eqref{eq:QvsGR},  it holds generically that $\SR(T ) \asymp \GR(T) = \Theta ( \Q(T)^2 )$.  To the best of our knowledge,  the upper bound $\SR(T) = O(\Q(T)^2)$ in Corollary~\ref{cor:QvsSR} is obtained for the first time.  

\subsubsection{Subrank vs.  analytic rank} 
By \cite[Theorem~3]{moshkovitz2024uniform} and \cite[Proposition~4.4]{chen2024stability},  we have $\GR(T) \asymp  \AR (T)$ for any tensor $T \in \mathbb{F}_q^{n_1} \otimes \cdots \otimes \mathbb{F}_q^{n_d}$. The following corollary is an immediate consequence of Theorem~\ref{3-GR<Q infinite}.
\begin{corollary}[$\AR(T)$ vs.  $\Q(T)$]\label{cor:QvsAR}
For any prime power $q$ and any tensor $T\in \mathbb{F}_q^{n_1} \otimes \mathbb{F}_q^{n_1}  \otimes  \mathbb{F}_q^{n_3}$,  we have $\AR(T) = O(\Q(T)^2)$.
\end{corollary}
For the lower bound of $\AR(T)$,  we have $\AR(T) = \Omega( \Q(T))$ by \cite[Theorem~5]{kopparty2020geometric}.  As far as we aware,  the quadratic upper bound of $\AR(T)$ in Corollary~\ref{cor:QvsAR} is established for the first time.   

\subsubsection{Gap of subrank} 
By \cite[Theorem~1]{CD21},  given a tensor $T \in \mathbb{K}^{n_1} \otimes \mathbb{K}^{n_2} \otimes \mathbb{K}^{n_3}$,  we either have $\SR( T^{\boxtimes k}) = 1$ for any $k \in \mathbb{N}$ or $\SR( T^{\boxtimes k}) \ge c^{k + o(k)}$ for any $k$ where $c \coloneqq (3/2)^{2/3} \approx 1.89$.  By Corollary~\ref{cor:QvsSR},  we obtain the following gap result for the subrank.
\begin{corollary}[Gap of $\Q(T)$]\label{cor:Qgap}
For any algebraically closed field $\mathbb{K}$ and any tensor $T \in \mathbb{K}^{n_1} \otimes \mathbb{K}^{n_2} \otimes \mathbb{K}^{n_3}$,  one of the following holds: 
\begin{enumerate}[label = (\alph*)]
\item $\Q(T^{\boxtimes k} ) = 1$ for all $k \in \mathbb{N}$. 
\item $\Q(T^{\boxtimes k} ) \ge c^{k/2 + o(k)}$ for all $k \in \mathbb{N}$.
\end{enumerate}
\end{corollary} 
While a stronger gap theorem \cite[Theorem~1.6]{CGZ23} can be derived via advanced geometric techniques,  by contrast,  Corollary~\ref{cor:Qgap} directly follows from existing results. 

\subsection{Organization of the paper}
For ease of reference,  we collect basic definitions and related facts in Section~\ref{sec:prelim}.  Then we prepare some technical results for polynomial ideals and the geometric rank respectively in Sections~\ref{sec:poly} and \ref{sec:geo},  which are of great importance in the proofs of the main theorems.  Main results of Section~\ref{sec:poly} are about the existence of special sequences in a polynomial ideal,  and the existence of rational points in a (quasi-)affine variety.  In Section~\ref{sec:geo},  we establish various equalities and inequalities for the geometric rank.  Lastly,  we provide complete proofs for Theorems~\ref{thm:general}, \ref{thm:general1} and \ref{3-GR<Q infinite} in Sections~\ref{sec:1.1},  \ref{sec:thm1.2} and \ref{sec:thm1.4}.

\section{Notations and preliminaries}\label{sec:prelim}
In this section,  we provide an overview of basic definitions and fundamental results in multilinear algebra,  commutative algebra,  and algebraic geometry.  
\subsection{Multilinear algebra}
Let $\mathbb{K}$ be a field.  Given $T \in \mathbb{K}^{n_1} \otimes \cdots \otimes \mathbb{K}^{n_d}$ and $S \in (\mathbb{K}^{n_{i_1}})^\ast \otimes \cdots \otimes (\mathbb{K}^{n_{i_s}})^\ast$ where $s \le d$ and $1 \le i_1 < \cdots < i_s \le d$,  we denote by $\langle T,  S \rangle$ the tensor in $\mathbb{K}^{n_{j_1}} \otimes \cdots \otimes \mathbb{K}^{n_{j_{d-s}}}$ obtained by contracting $T$ with $S$.  Here $\{  j_1 < \cdots < j_{d-s} \} = [n] \setminus \{i_1,\dots,  i_s\}$.  We denote by $\Hom ( (\mathbb{K}^{n_1})^\ast \times \cdots \times (\mathbb{K}^{n_d})^\ast,  \mathbb{K})$ the space of $\mathbb{K}$-multilinear functions on $(\mathbb{K}^{n_1})^\ast \times \cdots \times (\mathbb{K}^{n_d})^\ast$.  Then there is a canonical isomorphism identifying tensors with multilinear functions:
\begin{equation}\label{eq:iso}
\mathbb{K}^{n_1} \otimes \cdots \otimes \mathbb{K}^{n_d} \simeq \Hom ( (\mathbb{K}^{n_1})^\ast \times \cdots \times (\mathbb{K}^{n_d})^\ast,  \mathbb{K}),\quad T \mapsto f_T
\end{equation}
where $f_T(S) \coloneqq \langle T,  S \rangle$ for any $S \in (\mathbb{K}^{n_1})^\ast \otimes \cdots \otimes (\mathbb{K}^{n_d})^\ast$. 

Given $T \in \mathbb{K}^{n_1} \otimes \cdots \otimes \mathbb{K}^{n_d}$ and $S \in \mathbb{K}^{m_1} \otimes \cdots \otimes \mathbb{K}^{m_d}$ with $m_j \le n_j$ for each $j \in [d]$,  we say $S$ is a \emph{restriction} of $T$,  denoted by $S \le T$, if there are linear maps $g_j \in \Hom(\mathbb{K}^{n_j},  \mathbb{K}^{m_j})$ such that $S = (g_1 \otimes \cdots \otimes g_d) (T)$,  where $g_1 \otimes \cdots \otimes g_d$ is the linear map from $ \mathbb{K}^{n_1} \otimes \cdots \otimes \mathbb{K}^{n_d}$ to $\mathbb{K}^{m_1} \otimes \cdots \otimes \mathbb{K}^{m_d}$ induced by $g_1,\dots,  g_d$.

\begin{definition}[Tensor ranks]\label{def:rank}
Let $\mathbb{K}$ be a field and let $T \in \mathbb{K}^{n_1}\otimes \cdots \otimes \mathbb{K}^{n_d}$ be a tensor.  We define 
\begin{itemize}
\item Subrank of $T$:
\[
\Q (T)  \coloneqq \max\{r\in\mathbb{N}:I_{r}\le T\},
\]
where $I_{r} \in (\mathbb{F}^{r})^{\otimes d}$ is the identity tensor. 
\item Border subrank of $T$ for algebraically closed $\mathbb{K}$:
\[
\underline{\Q}(T) \coloneqq \max\{r \in \mathbb{N}:  I_r \in \overline{ G \cdot T } \},
\] 
where $G \coloneqq \GL_{n_1} (\mathbb{K}) \times \cdots \times \GL_{n_d} (\mathbb{K})$,  $G \cdot T$ denotes the orbit of $T$ under the action of $G$,  and $\overline{ G \cdot T }$ is the Zariski closure of $G \cdot T$.
\item Partition rank of $T$: 
\[
\PR(T) \coloneqq \min \left\lbrace
\sum_{J \subsetneq [d]} \rank(T_J): T = \sum_{J \subsetneq [d]} T_J,  \; T_J \in \mathbb{K}^{n_J \times n_{J^c} }
\right\rbrace.
\]
Here for each $J \subsetneq [d]$,  $\mathbb{K}^{n_J \times n_{J^c}}$ is the space of $\left(\prod_{j\in J} n_j\right) \times \left(\prod_{j\in J^c} n_j\right)$ matrices,  and in the equation $T = \sum_{J \subsetneq [d]} T_J$,  $T_J$ is regarded as a tensor in $\mathbb{K}^{n_1}\otimes \cdots \otimes \mathbb{K}^{n_d}$ via the isomorphism $\mathbb{K}^{n_J \times n_{J^c}} \simeq \mathbb{K}^{n_1}\otimes \cdots \otimes \mathbb{K}^{n_d}$.
\item Slice rank of $T$:
\[
\SR(T) \coloneqq \min \left\lbrace
\sum_{j=1}^d \rank(T_j): T = \sum_{j= 1}^d T_j,  \; T_j \in \mathbb{K}^{n_j \times \left( \prod_{i \ne j}n_i \right) }
\right\rbrace.
\]
\item Analytic rank of $T$ for $\mathbb{K} = \mathbb{F}_q$: 
\[
\AR(T) \coloneqq - \log_q \bias(T),
\] 
where 
\begin{equation}\label{eq:bias}
\bias(T) \coloneqq \Prob_{u_{i} \in \left( \mathbb{F}_{q}^{n_{i}} \right)^\ast,\;2 \le i\le d} \left( \left\langle T,   u_{2} \otimes \cdots \otimes u_{d} \right\rangle  =0 \right).
\end{equation}
\item Geometric rank of $T$:
\[
\GR(T) \coloneqq \codim\{(v_{1},\cdots,v_{d-1})\in(\mathbb{K}^{n_1})^{*}\times \cdots \times (\mathbb{K}^{n_d})^{*}: \langle T,  v_{1}\otimes \cdots \otimes v_{d-1} \rangle =0\}.
\]
\end{itemize} 
\end{definition}

\begin{lemma}\cite[Lemma~3.3]{briet2023discreteness}\label{3-minGR} 
Let $T \in \mathbb{K}^{n_{1}}\otimes\mathbb{K}^{n_{2}}\otimes\mathbb{K}^{n_{3}}$ and let $v_1,\dots,  v_c \in (\mathbb{K}^{n_3})^\ast$.  Denote $T_i \coloneqq \langle T,  v_i \rangle$,  $ i \in [c]$.  If $T_1,\dots,  T_c$ are linearly independent and $\rank (S) \ge 2c(c-1)$ for each $S \in \spa_{\mathbb{K}}( T_1,\cdots,T_{c})\setminus \{0\}$,  then $\Q(T)\ge c$. 
\end{lemma}

\begin{theorem}\cite{baily2024strength,chen2024stability,moshkovitz2024uniform} \label{AR stability}
There are functions $C_{1}$ and $C_{2}$ on $\mathbb{N}$ such that  for any $T \in \mathbb{F}_q^{n_1}\otimes \cdots \otimes \mathbb{F}_q^{n_d}$,  we have 
    \[
    C_1(d) \AR ( T )\le \GR( T )\le C_2(d)\AR ( T ).
    \]
\end{theorem}
\subsection{Commutative algebra} 
Let $R$ be a Noetherian ring and let $\mathfrak{p}$ be a prime ideal in $R$.  
The \emph{height} of $\mathfrak{p}$ is $\height(\mathfrak{p}) \coloneqq \max \left\lbrace h: 
\mathfrak{p}_0 \subsetneq \cdots \subsetneq \mathfrak{p}_h = \mathfrak{p}
\right\rbrace$.  For an arbitrary ideal $\mathfrak{a} \subseteq R$,  its \emph{height} is defined as $\height(\mathfrak{a}) \coloneqq \min \left\lbrace \height(\mathfrak{p}): \mathfrak{a} \subseteq \mathfrak{p},\;  \mathfrak{p} \text{~is prime}
\right\rbrace$.
\begin{theorem}\cite[Theorem~7.5]{kemper2011course}\label{thm:principal ideal}
For any ideal $\mathfrak{a} = (f_1,\dots,  f_n)$ in a Noetherian ring $R$,  we have $\height(\mathfrak{a}) \le n$. 
\end{theorem}

A Noetherian ring $R$ is said to satisfy \emph{the Serre's condition $(R_1)$} if $R_{\mathfrak{p}}$ is regular for each prime ideal $\mathfrak{p} \subseteq R$ of height at most $1$.  A sequence $f_1,\dots,  f_m \in  \mathbb{K}[x_1,\dots,  x_n]$ is called an \emph{$R_1$-sequence} (resp.  a prime sequence) if for each $1 \le s \le m$,  the quotient ring $\mathbb{K}[x_1,\dots,  x_n]/(f_1,\dots,  f_s)$ satisfies the Serre's condition $(R_1)$ (resp.  is a domain).  Clearly,  an $R_1$-sequence of forms in $\mathbb{K}[x_1,\dots,  x_n]$ is a prime sequence.  In particular,  it is a regular sequence. 

\begin{theorem}\cite[Theroem~B]{ananyan2020small}\label{algebraically closed small subalgebra}
There is a function $D(d,m)$ on $\mathbb{N} \times \mathbb{N}$ with the following property.  For any algebraically closed field $\mathbb{K}$ and forms $f_1,\dots,  f_m \in \mathbb{K}[x_1,\dots,  x_n]_{\le d}$,  there exists an $R_1$-sequence of forms $g_1,\dots,  g_s \in \mathbb{K}[x_1,\dots,  x_n]_{\le d}$,  where $s \le D(d,  m)$,  such that $\{ f_1,\dots,  f_m \} \subseteq \mathbb{K}[g_{1},\cdots,g_{m}]$. 
\end{theorem}
\begin{definition}[Strength]
For each nonzero form $f \in \mathbb{K}[x_1,\dots,  x_n]$,  the strength of $f$ is 
\[
\str(f) \coloneqq \min \left\lbrace
r: f = \sum_{i=1}^r a_i b_i \text{~for some forms $a_i,  b_i$ of positive degrees} 
\right\rbrace.
\]
By convention,  we have $\str(f) = \infty$ if $\deg (f) = 1$.  
\end{definition}

Given a form $f\in \mathbb{K}[x_1,\dots,  x_n]$,  we define the homogeneous ideal
\begin{equation}\label{eq:af}
\mathfrak{a}(f) \coloneqq (\partial f/ \partial x_1,\dots,  \partial f/ \partial x_n) \subseteq \mathbb{K}[x_1,\cdots,  x_n]. 
\end{equation}
We observe that $f \in \mathfrak{a}(f)$ as $f$ is a form.  Given a graded subspace $V \subseteq \mathbb{K}[x_1,\dots,  x_n]$,  we define 
\begin{equation}\label{eq:ht(af)}
\height( \mathfrak{a}(V) ) \coloneqq \min\{ \height (\mathfrak{a}(f)): \text{$f \in V\setminus \{0\}$ is a form} \},  \quad 
\str(V) \coloneqq \min\{ \str(f): \text{$f \in V\setminus \{0\}$ is a form} \}.
\end{equation}
In this paper,  we always view $V^{\overline{\mathbb{K}}} \coloneqq V \otimes_{\mathbb{K}} \overline{\mathbb{K}}$ as a $\overline{\mathbb{K}}$-vector space.  Thus,  there is no risk of confusion to use notations $\height( \mathfrak{a}(V^{\overline{\mathbb{K}}}) )$ and $\str(V^{\overline{\mathbb{K}}})$,  where
\begin{align*}
\height( \mathfrak{a}(V^{\overline{\mathbb{K}}}) ) &\coloneqq \min\{ \height (\mathfrak{a}(g)): \text{$g \in V^{\overline{\mathbb{K}}} \setminus \{0\}$ is a form} \},  \\ 
\str(V^{\overline{\mathbb{K}}}) &\coloneqq \min\{ \str(g): \text{$g \in V^{\overline{\mathbb{K}}} \setminus \{0\}$ is a form} \}.
\end{align*}
Here $\height(\mathfrak{a}(g))$ is the height (resp.  strength) of $\mathfrak{a}(g)$ (resp.  $g$) as an ideal (resp.  a form) in $\overline{\mathbb{K}}[x_1,\dots,  x_n]$.
\begin{theorem}\cite[Theorem~A]{ananyan2020small}  \label{codim vs str} 
There exists a function $E$ on $\mathbb{N} \times \mathbb{N}$ with the following property.  Suppose that $\mathbb{K}$ is an algebraically closed field and $f \in \mathbb{K}[x_1,\dots,  x_n]_{\le d}$ is a form such that $\height(\mathfrak{a}(f)) \le s$,  then $\str(f) \le E(d,s)$. 
\end{theorem}

\begin{theorem}\cite[Theorem~1.5.3]{bik2025strength}\label{stability of minstr}
There is a function $F$ on $\mathbb{N} \times \mathbb{N} \times \mathbb{N}$ with the following property.  Suppose $\mathbb{K}$ is a field and $V \subseteq \mathbb{K}[x_1,\dots,  x_n]_d$ is a linear subspace such that 
\begin{enumerate}[label=(\roman*)]
\item Either $\ch(\mathbb{K}) = 0$ or $\ch(\mathbb{K}) > d$.
\item $\dim V \le m $.
\item $\str(V^{\overline{\mathbb{K}}}) \le r$.
\end{enumerate}
Then $\str(V) \le F(d,m,r)$.
%
\end{theorem}

\subsection{Algebraic geometry}
\begin{lemma}\cite[Proposition~1.7.2]{stanley2011enumerative}\label{counting subspace} 
We denote by $\Gr(c,  \mathbb{F}_{q}^{n})$ the Grassmannian of $c$-dimensional linear subspaces of $\mathbb{F}_{q}^{n}$.  Then 
\[
\left\lvert \Gr(c,  \mathbb{F}_{q}^{n}) \right\rvert  = \frac{(q^{n}-1)\cdots(q^{n}-q^{c-1})}{(q^{c}-1)\cdots(q^{c}-q^{c-1})}.
\]
\end{lemma}

\begin{proposition}[Bezout theorem]\cite[Example~8.4.6]{fulton2013intersection} \label{gen-Bezout}
Let $\mathbb{K}$ be a field and let $X,  Y$ be affine subvarieties in $\mathbb{K}^n$.  Then we have $\deg(X \cap Y)\le \deg(X)\deg(Y)$.
\end{proposition}

\section{Properties of polynomial ideals}\label{sec:poly}
This section records some basic properties of polynomial ideals,  which are essential for proving Theorems~\ref{thm:general},  \ref{thm:general1} and \ref{3-GR<Q infinite}.  Although some results could be generalized to ideals in arbitrary Noetherian rings,  we focus on polynomial ideals for concreteness.  

We first investigate the existence of sequences of polynomials with special properties.  The main results are Propositions~\ref{long regular sequence-2},  \ref{linear extended},  \ref{mincodim and R_eta-algebraically closed} and \ref{mincodim and R_eta-general field}.  Before we proceed,  we recall the following formula that relates the height and dimension of a polynomial ideal.
\begin{lemma}\label{lem:dim+ht}
Let $\mathbb{K}$ be a field and let $\mathfrak{a} \subseteq R \coloneqq \mathbb{K}[x_1,\dots,  x_n]$ be an ideal.  We have 
\begin{equation}\label{eq:dim+ht}
\dim \left( R/\mathfrak{a} \right) + \height \left( \mathfrak{a} \right) = n.
\end{equation}
\end{lemma}
\begin{proof}
Let $\mathfrak{p}$ be a minimal prime ideal containing $\mathfrak{a}$ such that $\height( \mathfrak{p} ) = \height( \mathfrak{a}  )$.  Then we have 
\[
n = \dim \left( R/\mathfrak{p} \right) + \height \left( \mathfrak{p} \right)  \le \dim \left( R/\mathfrak{a} \right) + \height \left( \mathfrak{a} \right) \le n,
\]
where the equality follows from \cite[Chapter~5,  \S~14]{Matsumura70} or \cite[Theorem~1.8A]{Hartshorne77} and the two inequalities are obtained by definition.
\end{proof}

In the sequel,  we will need the following alternative characterization of regular sequences in polynomial rings.   We recall that polynomial rings are Cohen-Macaulay \cite[Exercise~2.1.17]{bruns1998cohen}.  For the definition and basic properties of Cohen-Macaulay rings,  one may refer to \cite[Section 16]{Matsumura70} or \cite[Section~2]{bruns1998cohen}.
\begin{lemma}[Regular sequence in polynomial ring]\label{lem:regular}
The polynomials $f_{1},\dots,  f_{m}\in \mathbb{K}[x_1,\dots,  x_n]$ form a regular sequence if and only if $\mathfrak{a} \coloneqq (f_{1},\cdots f_{m})$ is proper and $\height(\mathfrak{a} )=m$. 
\end{lemma}
\begin{proof}
Suppose $f_1,\dots,  f_m$ is a regular sequence.  Then Theorem~\ref{thm:principal ideal} together with \cite[Proposition~1.2.14]{bruns1998cohen} implies $\height(\mathfrak{a}) = m$.  Conversely,  if $\mathfrak{a} = (f_1,\dots,  f_m)$ is a proper ideal with $\height(\mathfrak{a}) = m$,  then we prove that $f_1,\dots,  f_m$ is a regular sequence by induction on $m$.  The claim is true for $m  = 1$ since $\mathbb{K}[x_1,\dots,  x_n]$ is an integral domain.  Assume the claim is true for $m  = k$ and we prove for $m = k + 1$.  According to Theorem~\ref{thm:principal ideal},  we must have $\height(\mathfrak{b})  = k$ where $\mathfrak{b} \coloneqq (f_1,\dots,  f_k)$.  The induction hypothesis implies that $f_1,\dots,  f_k$ is a regular sequence.  Since the polynomial ring is Cohen-Macaulay, $\mathfrak{b}$ is unmixed by \cite[Theorem~2.16]{bruns1998cohen}.  Therefore,  $f_{k+1}$ does not lie in any minimal prime of $\mathfrak{b}$ as $\height(\mathfrak{a}) > \height(\mathfrak{b})$.  In particular,  $f_1, \dots,  f_{k+1}$ is a regular sequence.
\end{proof}

Next,  we establish Lemmas~\ref{long regular sequence-1}--\ref{uncover},  which will be used to prove Propositions~\ref{long regular sequence-2} and \ref{linear extended}.  
\begin{lemma}\label{long regular sequence-1} 
Suppose that $\mathbb{K}$ is an infinite field.  Let $f_1,\dots,  f_m \in \mathbb{K}[x_1,\dots,  x_n]$ be polynomials such that $\mathfrak{a} \coloneqq (f_1,\dots,  f_m)$ is a proper ideal of height $s$.  There exist $g_1,\dots,  g_s \in \spa_{\mathbb{K}} \lbrace f_1,\dots,  f_m \rbrace$ such that $\height( \mathfrak{b} ) = s$ where $\mathfrak{b} = (g_1,\dots ,g_s)$.
\end{lemma}
\begin{proof}
Denote $V \coloneqq \spa_{\mathbb{K}} \{f_1,\dots,  f_m\}$ and $R \coloneqq \mathbb{K}[x_1,\dots,  x_n]$.  According to \cite[Corollary~3.2]{adiprasito2021schmidt} and Lemma~\ref{lem:dim+ht},  there exist $g_1,\dots,  g_s \in V$ such that $\dim ( R /\mathfrak{b}) = \dim ( R /\mathfrak{a}) = n - s$,  where $\mathfrak{b} = (g_1,\dots,  g_s)$.  By Lemma~\ref{lem:dim+ht} again,  we have $\height(\mathfrak{b}) = s$.
\end{proof}

Let $\mathbb{K}$ be a field and let $\mathfrak{a}$ be an ideal of $\mathbb{K}[x_1,\dots,  x_n]$.  We define 
\begin{equation}\label{eq:min-ideal}
\Min(\mathfrak{a}) \coloneqq \{ \mathfrak{p}: \text{minimal prime ideal over $\mathfrak{a}$} \}, \quad 
\min(\mathfrak{a}) \coloneqq \{ \mathfrak{p}\in \Min(\mathfrak{a}): \height(\mathfrak{p}) = \height(\mathfrak{a}) \}.
\end{equation}
Moreover,  we denote 
\[
\mathrm{Z}(\mathfrak{a}) \coloneqq \left\lbrace  
x \in \overline{\mathbb{K}}^n: f(x) = 0,\; f\in \mathfrak{a}
\right\rbrace.
\]
Recall that $\deg \mathrm{Z}(\mathfrak{a}) \coloneqq \sum_{j=1}^m \deg \mathrm{Z}_j$,  where $\mathrm{Z}_1,\dots,  \mathrm{Z}_m$ are irreducible components of $\mathrm{Z}(\mathfrak{a})$.
\begin{lemma}\label{simple observation}
Let $\mathbb{K}$ be a field and let $\mathfrak{a}$ be an ideal of $\mathbb{K}[x_1,\dots,  x_n]$.  We denote by $\mathfrak{a}^{\overline{\mathbb{K}}} = \mathfrak{a} \otimes_{\mathbb{K}} \overline{\mathbb{K}}$.  Then we have the following:
    \begin{enumerate}[label=(\alph*)]
    \item $\height(\mathfrak{a}^{\overline{K}}) = \height(\mathfrak{a})$.  \label{simple observation:item1}
    \item For any ideal $\mathfrak{b} \subseteq \mathfrak{a}$ of $\mathbb{K}[x_1,\dots,  x_n]$ such that $\height(\mathfrak{b})= \height( \mathfrak{a})$,  we have $\min( \mathfrak{a} ) \subseteq \min( \mathfrak{b} )$.   \label{simple observation:item2}
       \item $\lvert \min( \mathfrak{a} )\rvert \le \left\lvert \min( \mathfrak{a}^{\overline{\mathbb{K}}} ) \right\rvert \le \left\lvert \Min( \mathfrak{a}^{\overline{\mathbb{K}}} ) \right\rvert \le \deg \mathrm{Z}(\mathfrak{a})$.   \label{simple observation:item3}
        
    \end{enumerate}
\end{lemma}
\begin{proof}
We first prove \ref{simple observation:item2}.  For any $\mathfrak{p} \in \min( \mathfrak{a} )$,  we notice that $\mathfrak{b}\subseteq  \mathfrak{a} \subseteq \mathfrak{p}$ and $\height (\mathfrak{b} )= \height( \mathfrak{a} )= \height( \mathfrak{p} )$.  This implies $\mathfrak{p} \in \min ( \mathfrak{b} )$ and $\min( \mathfrak{a} )\subseteq \min( \mathfrak{b} )$.

For \ref{simple observation:item1} and \ref{simple observation:item3},  we denote $R \coloneqq \mathbb{K}[x_1,\dots,  x_n]$ and $\overline{R} \coloneqq \overline{\mathbb{K}}[x_1,\dots,  x_n]$.  By definition,  we have $\mathfrak{a}^{\overline{\mathbb{K}}} = \mathfrak{a} \overline{ R}$.  The last two inequalities in \ref{simple observation:item3} are obvious,  so it is sufficient to prove \ref{simple observation:item1} and $\left\lvert \min( \mathfrak{a} )\right\rvert \le \left\lvert \min( \mathfrak{a} \overline{ R} ) \right\rvert$.  We observe that $ \overline{R}$ is an integral extension of $R$ and $R$ is integrally closed in its fraction field.  Thus,  both the going-up and the going-down theorems hold for the extension $R \subseteq  \overline{ R}$ \cite[Theorem~9.4]{Matsumura86}.  According to \cite[Corollary~8.14]{kemper2011course},  for any prime ideal $\mathfrak{P} \subseteq \overline{R}$,  we have $\height( \mathfrak{P}  ) = \height( \mathfrak{P} \cap R)$.

Given any $\mathfrak{P} \in \min( \mathfrak{a} \overline{R})$,  it holds that $\height(\mathfrak{a}) \le \height( \mathfrak{P} \cap R ) = \height(\mathfrak{P}) = \height( \mathfrak{a} \overline{R} )$.  On the other side,  for any $\mathfrak{p} \in \min (\mathfrak{a})$,  there exists a prime ideal $\mathfrak{P}_0 \subseteq \overline{R}$ containing $\mathfrak{a} \overline{R }$ such that $\mathfrak{P}_0 \cap R= \mathfrak{p}$ by \cite[Theorem~8.12]{kemper2011course}.  Since $\mathfrak{a} \overline{R} \subseteq \mathfrak{p} \overline{R} \subseteq \mathfrak{P}_0$,  we conclude that $\height( \mathfrak{a} \overline{R}  ) \le \height( \mathfrak{P}_0 ) = \height( \mathfrak{p} ) = \height( \mathfrak{a} )$.  Therefore,  $\height( \mathfrak{a}) = \height(\mathfrak{a} \overline{R})$ for any ideal $\mathfrak{a} \subseteq R$ and this proves \ref{simple observation:item1}.

We claim that for any $\mathfrak{p} \in \min(\mathfrak{a})$,  there is some $\mathfrak{P} \in \min( \mathfrak{a} \overline{R})$ such that $\mathfrak{p} = \mathfrak{P} \cap R$.  This implies $\left\lvert \min (\mathfrak{a}) \right\rvert \le  \left\lvert \min (\mathfrak{a} \overline{R}) \right\rvert$.  To prove the claim,  we take $\mathfrak{P} \in \min( \mathfrak{p} \overline{R})$.  We notice that $\mathfrak{P} \in \min( \mathfrak{a} \overline{R})$,  as $\height( \mathfrak{a} \overline{R})= \height( \mathfrak{a} )= \height( \mathfrak{p} )= \height( \mathfrak{p}\overline{R})= \height ( \mathfrak{P})$.  Moreover,  since $\height( \mathfrak{P} \cap R ) = \height(\mathfrak{P}) = \height(\mathfrak{p})$ and $\mathfrak{p} \subseteq \mathfrak{P} \cap R$,  the going-down theorem forces the equality $\mathfrak{p} = \mathfrak{P} \cap R$.  
\end{proof}

\begin{lemma}\label{min prime bound}
Suppose $ \mathfrak{a} = (f_{1},\dots,f_{m}) \subseteq \mathbb{K}[x_{1},\dots,x_{n}]$.  If $\height( \mathfrak{a} )= s$ and $\deg(f_{i}) \le d$ for each $1 \le i \le m$,  then $\left\lvert \min( \mathfrak{a} ) \right\rvert\le d^s$.
\end{lemma}
\begin{proof}
By Lemma~\ref{simple observation},  we may assume that $\mathbb{K}$ is algebraically closed.  According to Lemma~\ref{long regular sequence-1}, there exist $g_{1},\cdots,g_{s} \in \spa_{\mathbb{K}}(f_{1},\cdots,f_{m})$ such that $\height( \mathfrak{b} )=s$,  where $\mathfrak{b} \coloneqq (g_{1},\cdots,g_{s})$.  Since $\mathfrak{b} \subseteq \mathfrak{a}$,  Lemma~\ref{simple observation} together with Lemma~\ref{gen-Bezout} implies that $\left\lvert \min( \mathfrak{a} ) \right\rvert \le \left\lvert \min(J) \right\rvert \le \deg( \mathrm{Z} ( \mathfrak{b} ) )\le d^{h}$.
\end{proof}   

In the sequel,  we will repeatedly invoke the following lemma.
\begin{lemma}\label{uncover}
Suppose $\mathbb{K}$ is a field with $\vert\mathbb{K}\vert\ge m+1$.  If  $S_1,\dots,  S_m$ are subsets of $\mathbb{K}^n$ such that $\bigcup_{i=1}^{m}S_{i}=\mathbb{K}^{n}$,  then $S_t$ contains a basis of $\mathbb{K}^n$ for some $1\le t\le m$.  
\end{lemma}
\begin{proof}
By assumption,  we have $\cup_{i=1}^{m} \spa_{\mathbb{K}}( S_{i} )=\mathbb{K}^{n}$.  It is sufficient to prove that $\spa_{\mathbb{K}}( S_t) = \mathbb{K}^n$ for some $1 \le t \le m$.  If $\mathbb{K}$ is an infinite field,  the proof is completed by dimension counting.  Assume that $\mathbb{K}$ is a finite field with $|\mathbb{K}| \ge m+1$ but $\spa_{\mathbb{K}}( S_i) \neq \mathbb{K}^n$ for each $1 \le i \le m$.  Then we may obtain a contradiction: 
\[
\lvert \mathbb{K} \rvert^n =  \lvert\mathbb{K}^{n}\rvert\le\sum_{i=1}^{m} \lvert \spa_{\mathbb{K}}( S_{i})\rvert\le m\lvert \mathbb{K} \rvert^{n-1} < \lvert \mathbb{K} \rvert^{n}.  \qedhere
\]
\end{proof}

With all the technical lemmas prepared,  we can now prove the main results of this section.
\begin{proposition}\label{long regular sequence-2} 
Let $\mathbb{K}$ be a field and let $\mathfrak{a} = (f_{1},\dots,f_{m})\subseteq \mathbb{K}[x_{1},\dots,x_{n}]$ be a proper ideal where $\deg(f_{i})\le d$ for each $1 \le i \le m$.  If $\height(\mathfrak{a}) = s$ and $\left\lvert \mathbb{K} \right\rvert \ge d^{s}+1$,  then there exist $g_{1},\dots,g_{s}\in \spa_{\mathbb{K}}(f_{1},\dots,f_{m})$ such that $\height( \mathfrak{b} )=s$ where $\mathfrak{b} \coloneqq (g_{1},\dots,g_{s})$.
\end{proposition}
\begin{proof}
Denote $V \coloneqq \spa_{\mathbb{K}}(f_{1},\cdots,f_{m})$.  We claim that there exists a finite subset $S = \{ g_1,\dots,  g_t \} \subseteq V$ such that
\begin{enumerate}[label=(\roman*)]
\item $S$ is a regular sequence. \label{long regular sequence-2:item1}
\item For any $g\in V$,  $g\in \mathfrak{p}$ for some $\mathfrak{p}\in \min(\mathfrak{b})$ where $ \mathfrak{b} \coloneqq (g_1,\dots,  g_t)$.\label{long regular sequence-2:item2}
\end{enumerate}
In fact,  we can construct $S$ as follows: Let $g_1$ be any element in $V \setminus \{0\}$.  We denote $S_1 \coloneqq \{g_1\}$ and $\mathfrak{b}_1 \coloneqq (g_1)$.  If there exists $g_2 \in V$ such that $\height(\mathfrak{b}_2) = \height(\mathfrak{b}_1) + 1$,  where $\mathfrak{b}_2 \coloneqq (g_1,g_2)$,  then $S_2 \coloneqq \{g_1,g_2\}$ is a regular sequence.  We continue this procedure to obtain regular sequences $S_3,  S_4,  \dots$.  The procedure stops at $S \coloneqq S_t$ for some $t \le s$,  since $\height(\mathfrak{a}) = s$.  By construction,  for any $g\in V$,  we have $\height( \mathfrak{b}' ) = \height( \mathfrak{b})$ where $\mathfrak{b}' \coloneqq (g_1,\dots,  g_t,  g)$ and $\mathfrak{b} \coloneqq (g_1,\dots,  g_t)$.  By Lemma~\ref{simple observation},  we have $\min (\mathfrak{b}') \subseteq \min (\mathfrak{b})$ and \ref{long regular sequence-2:item2} is satisfied.

It is left to prove that $s \le t$.  According to \ref{long regular sequence-2:item2},  we have $V = \cup_{ \mathfrak{p} \in \min( \mathfrak{b} )}( \mathfrak{p} \cap  V )$.  Since $|S| = t \le s$,  Lemma~\ref{min prime bound} implies that $\left\lvert \min(\mathfrak{b}) \right\rvert \le d^{s} < |\mathbb{K}|$.  By Lemma~\ref{uncover},  there is $\mathfrak{q} \in \min(\mathfrak{b})$ such that $\mathfrak{q} \cap V = V$.  In particular,  we obtain $\mathfrak{a} \subseteq \mathfrak{q}$ which implies $s = \height(\mathfrak{a}) \le \height (\mathfrak{q}) =  \height (\mathfrak{b}) \le t$.  Here,  the last inequality follows from Theorem~\ref{thm:principal ideal}.
\end{proof}
\begin{remark}
If $\mathbb{K}$ is an infinite field,  Proposition~\ref{long regular sequence-2} reduces to Lemma~\ref{long regular sequence-1} immediately.   However,  we want to emphasize that the proof of Proposition~\ref{long regular sequence-2} depends on Lemma~\ref{min prime bound}, which in turn relies on Lemma~\ref{long regular sequence-1}.
\end{remark}

\begin{proposition}\label{linear extended}
Let $\mathbb{K}$ be a field and let $f_{1},\cdots,f_{m} \in \mathbb{K}[x_{1},\cdots,x_{n}]$ be forms satisfying \begin{enumerate}[label=(\roman*)]
\item $d^{m}+1 \le \left\lvert \mathbb{K}\right\rvert$ and $\deg (f_i) \le d$,  $1 \le i \le m$. 
\item $f_{1},\cdots,f_{m}$ form a regular sequence.
\end{enumerate}
Then there exist linear forms $l_{i}$,  $1\le i\le n-m$,  such that $f_{1},\cdots,f_{m}, l_{1},\cdots, l_{n-m}$ is a regular sequence.
\end{proposition}
\begin{proof}
Let $\{l_1,\dots,  l_c\}$ be a maximal set of linear forms such that $f_1,\dots,  f_m,  l_1,\dots,  l_c$ is a regular sequence,  where $0 \le c \le n-m$.  We claim that $c = n - m$.  It is sufficient to prove $\height(\mathfrak{a}) = n$,  where $\mathfrak{a} \coloneqq (f_1,\dots,  f_m,  l_1,\dots,  l_c)$.  Indeed,  if $c < n - m$,  then by Theorem~\ref{thm:principal ideal},  we must have $\height(\mathfrak{a}) < n$.

To complete the proof,  we show that there is some $\mathfrak{q}\in \min(\mathfrak{a})$ such that $\height(\mathfrak{q}) = n$.  Let $V$ be the space of linear forms.  By the maximality of $\{l_1,\dots,  l_c\}$,  for each $l \in V$,  it holds that $\height(\mathfrak{b}) = \height(\mathfrak{a})$,  where $\mathfrak{b} = \mathfrak{a} + (l)$.  According to Lemma~\ref{simple observation},  we obtain $\min(\mathfrak{b}) \subseteq \min (\mathfrak{a})$,  which implies $V = \cup_{\mathfrak{p} \in \min(\mathfrak{a})} ( \mathfrak{p} \cap V)$.  By Lemma~\ref{simple observation} and Proposition~\ref{gen-Bezout},  we also have $\left\lvert \mathfrak{a} \right\rvert \le d^m < \left\lvert \mathbb{K} \right\rvert$.  Thus,  we may derive from Lemma~\ref{uncover} that $V \subseteq \mathfrak{q}$ for some $\mathfrak{q} \in \min (\mathfrak{a})$.  Consequently,  $\mathfrak{q} = (x_1,\dots,  x_n)$ and $\height(\mathfrak{a}) = \height(\mathfrak{q}) = n$.
\end{proof}

Given a form $f \in \mathbb{K}[x_1,\dots,  x_n]$,  we recall from \eqref{eq:af} that $\mathfrak{a}(f)$ is the ideal in $\mathbb{K}[x_1,\dots,  x_n]$ generated by partial derivatives of $f$.  Moreover,  a sequence $f_1,\dots,  f_m \in \mathbb{K}[x_1,\dots,  x_n]$ is called an $R_1$-sequence if for each $1 \le s \le m$,  $\mathbb{K}[x_1,\dots,  x_n]/(f_1,\dots,  f_s)$ satisfies the Serre's condition $(R_1)$.  In particular,  an $R_1$-sequence of forms in $\mathbb{K}[x_1,\dots,  x_n]$ is a regular sequence.  The lemma that follows is extracted from the proof of (b) of Theorem~A in \cite{ananyan2020small}. 
\begin{lemma}\label{mincodim and R_eta-algebraically closed}
There exists a function $G$ on $\mathbb{N} \times \mathbb{N}$ with the following property.  Suppose $\mathbb{K}$ is a field and $V \subseteq \mathbb{K}[x_1,\dots,  x_n]$ is a subspace such that 
\begin{enumerate}[label=(\roman*)]
\item $\mathbb{K}$ is algebraically closed.
\item $V$ is graded and $\dim_{\mathbb{K}} V \le m$. \label{mincodim and R_eta-algebraically closed:item1}
\item $\deg(f) \le d$ for each form $f \in V$.\label{mincodim and R_eta-algebraically closed:item2}
\item $\height(\mathfrak{a}(V)) \ge G(d,m)$. \label{mincodim and R_eta-algebraically closed:item3}
\end{enumerate}
Then every $\mathbb{K}$-linear basis of $V$ is an $R_1$-sequence.
\end{lemma}
\begin{remark}
The statement of \cite[Theorem~A~(b)]{ananyan2020small} includes the assumption that $\str(f) \ge G(d,m)$ for each form $f\in V \setminus \{0\}$.  This makes Lemma~\ref{mincodim and R_eta-algebraically closed} appear distinct from \cite[Theorem~A~(b)]{ananyan2020small} at first glance.  However,  a closer investigation of the proof of \cite[Theorem~A~(b)]{ananyan2020small} reveals that the strength condition actually implies \ref{mincodim and R_eta-algebraically closed:item3} of Lemma~\ref{mincodim and R_eta-algebraically closed} by inductively applying \cite[Theorem~2.5]{ananyan2020small},  from which the desired conclusion follows. 
\end{remark} 
 
\begin{lemma}\label{lem:str-height}
Let $\mathbb{K}$ be a field.  For each form $f \in \mathbb{K}[x_1,\dots,  x_n]$,  we have $2 \str(f) \ge \height(\mathfrak{a}(f))$.
\end{lemma} 
\begin{proof}
Suppose $\str(f) = s$ so that $f = \sum_{k=1}^s g_k h_k$ for some polynomials $g_1,\dots,  g_s,  h_1,\dots,  h_s \in \mathbb{K}[x_1,\dots,  x_n]$. Since $\partial f / \partial x_j = \sum_{k = 1}^s \left( h_k \partial g_k/ \partial x_j  + g_k \partial h_k/ \partial x_j \right)$,  we have 
\[
\mathfrak{a}(f)  \subseteq \mathfrak{b} \coloneqq (g_1,\dots,  g_s,  h_1,\dots,  h_s).\]
By Theorem~\ref{thm:principal ideal}, we obtain $\height(\mathfrak{a}(f)) \le \height(\mathfrak{b}) \le 2s$.
\end{proof}
 
Given a graded subspace $V\subseteq \mathbb{K}[x_1,\dots,  x_m]$,  we recall from \eqref{eq:ht(af)} that 
\begin{align*} 
\height( \mathfrak{a}(V) ) &\coloneqq \min\{ \height (\mathfrak{a}(f)): \text{$f$ is a form in $V\setminus \{0\}$} \},  \\
\str(V) &\coloneqq \min\{ \str(f): \text{$f$ is a form in $V\setminus \{0\}$} \}.
\end{align*}
Moreover,  $\height( \mathfrak{a}(V^{\overline{\mathbb{K}}}) )$ and $\str(V^{\overline{\mathbb{K}}})$ are defined by viewing $V^{\overline{\mathbb{K}}}$ as a vector space over $\overline{\mathbb{K}}$.
\begin{proposition}\label{mincodim and R_eta-general field}
There is a function $H$ on $\mathbb{N} \times \mathbb{N}$ satisfying the following property.  Assume $\mathbb{K}$ is a field and $V \subseteq \mathbb{K}[x_1,\dots,  x_n]$ is a subspace such that 
\begin{enumerate}[label=(\roman*)]
\item Either $\ch(\mathbb{K}) = 0$ or $\ch(\mathbb{K}) > d$. 
\item $V$ is graded and $\dim V \le m$. 
\item $\deg (f) \le d$ for each form $f\in V\setminus \{0\}$
\item $\height(\mathfrak{a}(V)) \ge H(d,m)$.
\end{enumerate}
Then every $\mathbb{K}$-linear basis of $V$ an $R_1$-sequence.
\end{proposition}
\begin{proof}
Denote $r \coloneqq \str(V) $ and $s \coloneqq \height(\mathfrak{a}(V))$. By Lemma~\ref{lem:str-height},  we have $r\ge s/2$. According to Theorem~\ref{stability of minstr},  there is a function $\widetilde{F}$ on $\mathbb{N} \times \mathbb{N} \times \mathbb{N}$ such that $\overline{r} \coloneqq \str(V^{\overline{\mathbb{K}}}) \ge \widetilde{F}(d,m,  r )$. Theorem~\ref{codim vs str} implies  $\height(\mathfrak{a}(V^{\overline{\mathbb{K}}})) \ge \widetilde{E}(d,  \overline{r} )$ for some function $\widetilde{E}$ on $\mathbb{N} \times \mathbb{N}$.  Putting all the above inequalities together,  we obtain a function $C$ on $\mathbb{N} \times \mathbb{N} \times \mathbb{N}$ such that $\height(\mathfrak{a}(V^{\overline{\mathbb{K}}})) \ge C(d, m, s)$.  Therefore,  we may choose $H(d,m)$ so that $C(d,m,H(d,m)) \ge G(d,m)$ where $G$ is the function in Lemma~\ref{mincodim and R_eta-algebraically closed}.  By Lemma~\ref{mincodim and R_eta-algebraically closed},  every $\overline{\mathbb{K}}$-linear basis of $V^{\overline{\mathbb{K}}}$ is an $R_1$-sequence in $\overline{\mathbb{K}}[x_1,\dots,  x_n]$.  In particular,  every $\mathbb{K}$-linear basis of $V$ is an $R_1$-sequence in $\mathbb{K}[x_1,\dots, x_n]$.
\end{proof}

The proof of the following proposition is inspired by that of \cite[Theorem~B]{ananyan2020small}.
\begin{proposition}\label{small subalgebra general field}
There is a function $I$ on $\mathbb{N} \times \mathbb{N}$ satisfying the following property. Assume $\mathbb{K}$ is a field and $f_1,\dots,  f_m \in \mathbb{K}[x_1,\dots,  x_n]$  are forms such that
\begin{enumerate}[label=(\roman*)]
\item Either $\ch(\mathbb{K}) = 0$ or $\ch(\mathbb{K}) > d$. 
\item For each $1 \le i \le m$,  $\deg(f_i) \le d$. 
\end{enumerate}
Then there exists an $R_1$-sequence of forms $g_1,\dots,  g_s \in \mathbb{K}[x_1,\dots,  x_n]$ such that $s \le I(d,m)$,  $\{f_1,\dots,  f_m\} \subseteq \mathbb{K}[g_1,\dots,  g_s]$ and $\deg(g_j) \le d$ for each $j \in [s]$. 
\end{proposition}
\begin{proof} 
Let $V = \oplus_{i=1}^d V_i$ be a graded subspace of $R \coloneqq \mathbb{K}[x_1,\dots,  x_n]$ and let $n_i \coloneqq \dim V_i$ for $i\in [d]$.  We claim that there exists an $R_1$-sequence of forms $g_1,\dots,  g_s\in R$ such that $s \le I(d,n_1 + \cdots + n_d)$ and $V \subseteq \mathbb{K}[g_1,\dots,  g_s]$.  

To prove the claim,  we denote $n(V) \coloneqq (n_1,\dots,  n_d) \in \mathbb{N}^d$ and equip $\mathbb{N}^d$ with the colexicographic order.  We proceed by induction on this order.  If $n(V) = (n_1,0,  \dots,  0)$,  there is nothing to prove.  Assume the claim is true for any graded subspace $W \subseteq R$ with $n(W)< (n_1,\dots, n_d)$.  Now we prove the claim for a graded subspace $V \subseteq R$ with $n(V)= (n_1,\dots, n_d)$.  We remark that although  the induction hypothesis already provides us a function $I$ on $\mathbb{N} \times \mathbb{N}$,  it may not work for $V$.  Thus,  we will need to modify $I$ in the induction step.  

Let $E,F$ and $H$ be functions in Theorems~\ref{codim vs str},  \ref{stability of minstr} and Proposition~\ref{mincodim and R_eta-general field} respectively,  and let $m \coloneqq n_1 + \cdots + n_d$. Suppose that $\height(\mathfrak{a}(V)) \ge H(d,  m)$, then Proposition~\ref{mincodim and R_eta-general field} implies that any basis $g_1,\dots,  g_{s}$ of $V$ is a desired sequence,  where $s \coloneqq \dim V \le m$.  If there is some form $f \in V$ such that $\height(\mathfrak{a}(f)) < H(d,  m)$,  we pick such an $f$ of the largest degree.  Then $\height(\mathfrak{a}(f)^{\overline{\mathbb{K}}}) < H(d,  m)$ by Lemma~\ref{simple observation}.  According to Theorem~\ref{codim vs str},  we conclude that $\str_{\overline{\mathbb{K}}}(f) \le E(d,  H(d,  m))$.  Thus,  we may deduce from Theorem~\ref{stability of minstr} that $r \coloneqq \str_{\mathbb{K}}(f) \le  F(d,1,E(d,H(d,m)))$.  We write $f = \sum_{i=1}^r g_i h_i$ for some $g_i, h_i \in R$ of degrees strictly less than $\deg(f)$.  Let $V'_t \subseteq V_t$ be a subspace such that $V_t = V'_t \oplus \spa_{\mathbb{K}} \{ f \}$.  We consider 
\[
V' \coloneqq \left(  V'_t \oplus \left( \oplus_{1 \le i \ne t \le d} V_i \right) \right) + \spa_{\mathbb{K}} \{g_1,\dots, g_r,  h_1,  \dots, h_r\}.
\]
By construction,  we have $V \subseteq V'$,  $n(V') < n (V) = (n_1,\dots,  n_d)$ and $\dim V' \le \dim V - 1 + 2r $.  By induction hypothesis,  there exists an $R_1$-sequence $g_1,\dots,  g_s \in R$ such that $V' \subseteq \mathbb{K}[g_1,\dots,  g_s]$ and 
\[
s \le I(d,  \dim V') \le I(d,  m - 1 + 2F(d,1,E(d,H(d,m)))).
\]
Updating $I(d,m)$ to $I(d,  m - 1 + 2F(d,1,E(d,H(d,m))))$,  the induction step is complete.  
\end{proof}

Next,  we study the existence of solutions of polynomial systems in field extensions.  In Lemma~\ref{solve equation} and Proposition~\ref{solution of constructable set} below,  we prove that under some conditions, a polynomial system over $\mathbb{K}$ always admits a solution over some extension $\mathbb{L}/\mathbb{K}$ with $[\mathbb{L}: \mathbb{K}]$ uniformly bounded. 
\begin{lemma}[Existence of rational point I]\label{solve equation}
There is a function $K(d,m)$ on $\mathbb{N} \times \mathbb{N}$ satisfying the following property.  Suppose $\mathbb{K}$ is a field and $f_1,\dots,  f_m \in \mathbb{K}[x_1,\dots,  x_n]$ are forms such that 
\begin{enumerate}[label=(\roman*)]
\item Either $\ch(\mathbb{K}) = 0$ or $\ch(\mathbb{K})> d$.
\item $\deg(f_i) \le d$ for each $1 \le i \le n$.
\item There is some $u \in \overline{\mathbb{K}}^n \setminus \{0\}$ such that $f_1(u) = \cdots = f_m(u) = 0$.
\end{enumerate}
Then there exists an extension $\mathbb{L}$ of $\mathbb{K}$ such that $[\mathbb{L}:\mathbb{K}] \le K(d,m)$ and $f_1(v) = \cdots = f_m(v) = 0$ for some $v \in \mathbb{L}^n \setminus \{0\}$.
\end{lemma}
\begin{proof}
We denote $R \coloneqq \mathbb{K}[x_1,\dots, x_n]$,  $\mathfrak{a} \coloneqq (f_1,\dots,  f_m)\subseteq R$ and $Z \coloneqq \Spec(R/\mathfrak{a})$.  The existence of some $v \in \mathbb{L}^n \setminus \{0\}$ such that $f_1(v) = \cdots = f_n(v) = 0$ is equivalent to $Z(\mathbb{L}) \setminus \{0\} \ne \emptyset$,  i.e.,  there is a $\mathbb{K}$-algebra homomorphism $\varphi: R/\mathfrak{a} \to \mathbb{L}$ such that $\varphi^{-1}(0) \not\equiv (x_1,\dots,  x_n) \pmod{\mathfrak{a}}$.

By Proposition~\ref{small subalgebra general field},  there is a function $I$ on $\mathbb{N} \times \mathbb{N}$ and a regular sequence $g_{1},\dots,g_{s}$ of forms in $\mathbb{K}[x_{1},\dots,x_{n}]$ such that $\mathfrak{a} \subseteq (g_{1},\cdots,g_{s})$,  $s \le I(d,m)$ and $\deg (g_i) \le d$ for each $1 \le i \le s$.  

Let $\mathbb{K}'$ be an extension of $\mathbb{K}$ such that $\vert\mathbb{K}'\vert\ge d^{I(d,m)}+1$ and $[\mathbb{K'}:\mathbb{K}]\le I(d,m)\log_{2}d$.  In fact,  if $\mathbb{K}$ is an infinite field,  then we may choose $\mathbb{K}' = \mathbb{K}$.  If $\mathbb{K} = \mathbb{F}_{p^r}$ where $p > d$,  then we choose $\mathbb{K}' = \mathbb{F}_{p^{r k}}$ where $k = I(d,m) \log_2 d$.  Since $\otimes_{\mathbb{K}} \mathbb{K}'$ is an exact functor,  $g_1,\dots, g_s$ is also a regular sequence in $\mathbb{K}'[x_1,\dots,  x_n]$. According to Proposition~\ref{linear extended},  there are linear forms $l_1,  \dots,  l_{n-s} \in R'\coloneqq \mathbb{K}'[x_1,\dots, x_n]$ such that $g_1,\dots,  g_s,  l_1,\dots,  l_{n-s}$ is a regular sequence.  

Denote $\mathfrak{a}' \coloneqq (g_1,\dots,  g_s,  l_1,\dots,  l_{n-s-1}) \subseteq R'$ and $X \coloneqq \Spec(R'/\mathfrak{a}')$.  By Lemmas~\ref{lem:dim+ht} and \ref{lem:regular},  we have $\height(\mathfrak{a}) = n-1$ and $\dim X = 1$.  Thus there is some $a = (a_1,\dots,  a_n) \in X (\overline{\mathbb{K}}) \setminus \{0\}$.  Suppose without loss of generality that $a_1 = 1$.  We consider $\mathfrak{b} \coloneqq \mathfrak{a}' + (x_1 - 1) \subsetneq R'$.  Since $\mathfrak{a}'$ is a homogeneous ideal,  so is every $\mathfrak{p} \in \min (\mathfrak{a}') = \Ass(\mathfrak{a}')$.  In particular,  $x_1 - 1 \not\in \mathfrak{p}$.  Thus,  $g_1,  \dots, g_s,  l_1,\dots,  l_{n-s-1},  x_1 - 1$ is a regular sequence,  which implies $\height(\mathfrak{b}) = n$.  Let $\mathfrak{m} \subsetneq R'$ be a maximal ideal containing $\mathfrak{b}$.  Then $\mathfrak{m} \ne (x_1,\dots,  x_n)$ and $\mathfrak{a}R' \subseteq (g_1,\dots,  g_s) R' \subseteq \mathfrak{b} \subseteq  \mathfrak{m}$.  Therefore,  we obtain an obvious homomorphism of $\mathbb{K}$ algebras:
\[
R/\mathfrak{a}  \to (R/\mathfrak{a})\otimes_{\mathbb{K}} \mathbb{K}' = R'/\mathfrak{a}R' \to  R'/\mathfrak{m} \eqcolon  \mathbb{L},
\]
which provides a nonzero $\mathbb{L}$-rational point of $Z$.  

It is left to estimate $[\mathbb{L}: \mathbb{K}]$.  To achieve this,  we observe that $[\mathbb{L}: \mathbb{K}] = [R'/\mathfrak{m}: \mathbb{K}'] [\mathbb{K}': \mathbb{K}]$.  By the choice of $\mathbb{K}'$,  we have $[\mathbb{K}': \mathbb{K}]  \le I(d, m) \log_2 d$.  Moreover,  let $Y = \Spec(R'/\mathfrak{b})$,  then $\mathfrak{b} \subseteq \mathfrak{m}$ together with Lemma~\ref{gen-Bezout} leads to
\[
[R'/\mathfrak{m}: \mathbb{K}'] \le \dim_{\mathbb{K}'}(R'/\mathfrak{b}) 
= \deg(Y)  \le d^s \le d^{I(d,m)}.
\] 
Therefore,  we have $[\mathbb{L}: \mathbb{K}] \le K(d,m) \coloneqq d^{I(d,m)} I(d,m) \log_2 d$.
\end{proof}
\begin{proposition}[Existence of rational point II]\label{solution of constructable set}
There is a function $L$ on $\mathbb{N} \times \mathbb{N}$ with the following property.  Suppose that $\mathbb{K}$ is a field and $f_1,\dots,  f_m,  g \in \mathbb{K}[x_1,\dots,  x_n]$ are forms such that 
\begin{enumerate}[label=(\roman*)]
\item $\mathbb{K}$ is algebraically closed,  or $\ch(\mathbb{K}) = 0$,  or $\ch(\mathbb{K})> d$.
\item $\deg g \le d$ and $\deg f_i \le d$ for each $1 \le i \le m+1$. 
\item $f_1,\dots, f_m$ form an $R_1$-sequence.  
\item $g \not\in (f_1,\dots,  f_m)$.
\end{enumerate} 
Then there is a field extension $\mathbb{L}$ of $\mathbb{K}$ such that $[\mathbb{L}:  \mathbb{K}] \le L(d,m)$ and $f_1(v) = \cdots = f_m (v) = 0$ but $g(v) = 1$ for some $v \in \mathbb{L}^n$.
\end{proposition}
\begin{proof}
By assumption,  $(f_{1},\cdots,f_m)$ is a prime ideal and $f_1,\dots,  f_m,  g$ form a regular sequence in $R \coloneqq \mathbb{K}[x_1,\dots, x_n]$.  In particular,  we must have $m \le n -1$.  If $\mathbb{K}$ is algebraically closed,  then $\mathbb{L} = \mathbb{K}$ and the existence of $v$ is obvious by the Hilbert's Nullstellensatz and the fact that $(f_{1},\cdots,f_{m})$ is prime.

Let $\mathbb{F}$ be an extension of $\mathbb{K}$ such that $\vert\mathbb{F}\vert\ge d^{m+1}+1$ and $[\mathbb{F}:\mathbb{K}]\le m+2$.  We observe that $f_1,\dots,  f_m,  g$ form a regular sequence in $R \coloneqq \mathbb{F}[x_1,\dots, x_n]$.  By Proposition~\ref{linear extended},  there exist linear forms $l_1,\dots,  l_{n- m -1} \in R$ such that $f_1,\dots,  f_m,  g,  l_1,\dots,  l_{n-m-1}$ is a regular sequence.  We observe that $R/(l_1,\dots,  l_{n-m-1}) \simeq R_0 \coloneqq \mathbb{F}[x_1,\dots,  x_{m+1}]$.  We denote by $\overline{f}_1,\dots,  \overline{f}_m,  \overline{g}$ the images of $f_1,\dots,  f_m,  g$ in $R_0$,  respectively.  Let $\mathfrak{a} \coloneqq (\overline{f}_1,\dots,  \overline{f}_m)$ and $\mathfrak{b} = \mathfrak{a} + (\overline{g})$. Then $\height(\mathfrak{a}) = m$ and $\height(\mathfrak{b}) = m+1$.

Therefore,  Lemma~\ref{solve equation} implies that  there is some $u \in \mathbb{L}_0^{m+1} \setminus \{0\}$ such that $\overline{f}_1(u) = \cdots = \overline{f}_m(u)$,  where $\mathbb{L}_0$ is an extension of $\mathbb{F}$ with $[\mathbb{L}_0: \mathbb{F}] \le K(d,m)$.  We claim that $\overline{g}(u) \ne 0$.  Otherwise,  $u$ satisfies $\overline{f}_1(u) = \cdots = \overline{f}_m(u) = \overline{g}(u) = 0$.  Since $\mathfrak{b}$ is homogeneous and has height $m+1=\dim(R_{0})$, we obtain $X \coloneqq \Spec (R_0/\mathfrak{b}) = \{ \mathfrak{m}_0 \}$, where $\mathfrak{m}_0 \coloneqq (x_1,\dots,  x_{m+1})$. This contradicts to the fact that $u \ne 0$. 

Using the isomorphism $R/(l_1,\dots,  l_{n-m-1}) \simeq R_0$,  we may easily convert $u \in \mathbb{L}_0^{m+1}$ to $w \in \mathbb{L}_0^n$ satisfying $f_1(w) = \cdots = f_m(w) = 0$ and $g(w) \ne 0$.  Since $\deg(g) \le d$,  there exists an extension $\mathbb{L} / \mathbb{L}_0$ of degree at most $d$ to ensure the existence of $\lambda \in \mathbb{L} \setminus \{0\}$ such that $g (v) = 1$ where $v \coloneqq \lambda w \in \mathbb{L}^n$.  By construction,  we have 
\[
[\mathbb{L} : \mathbb{K}] = [\mathbb{L}:\mathbb{L}_0][\mathbb{L}_0 : \mathbb{F}] [\mathbb{F} : \mathbb{K}] \le d(m+2)K(d,m) \eqqcolon L(d,m).  \qedhere
\]
\end{proof}
\begin{remark}
Let $u = (u_1,\dots,  u_n) \in \overline{\mathbb{K}}^n$ be such that $f_1(u) = \cdots = f_m (u)$ but $g(u) = 1$.  It is clear that $u \in \mathbb{F}^n$ and $[\mathbb{F} : \mathbb{K}] < \infty$,  if we take $\mathbb{F} \coloneqq \mathbb{K}(u_1,\dots,  u_n)$.  A priori,  however,  $[\mathbb{F}: \mathbb{K}]$ may depend on $f_1,\dots,  f_m$ and $g$.  Proposition~\ref{solution of constructable set} exhibits an upper bound for $[\mathbb{F}: \mathbb{K}]$,  which only depends on $m$ and the degrees of $f_1,\dots,  f_m$ and $g$. 
\end{remark}

The following lemma is about the generators of polynomial ideals under field extensions.  We provide a proof due to the lack of appropriate references.
\begin{lemma}\label{infinite generate}
Let $S$ be a subset of $\mathbb{K}^{n}$ and let $\mathfrak{a} \coloneqq \left\lbrace
f \in \overline{\mathbb{K}}[x_1,\dots,  x_n]: f(u) = 0,\; u \in S
\right\rbrace$. Then $\mathfrak{a}$ is generated by polynomials in $\mathbb{K}[x_{1},\cdots,x_{n}]$.
\end{lemma}
\begin{proof}
Suppose $g_{1},\cdots,g_{m}$ are polynomials in $\overline{\mathbb{K}}[x_{1},\cdots,x_{n}]$ such that $\mathfrak{a} = (g_1,\dots,  g_m)$.  Then there are $c_j \in \overline{\mathbb{K}}$ and $f_{ij} \in \mathbb{K}[x_1,\dots,  x_n]$,  where $1 \le i \le m$ and $1 \le j \le s$,  such that 
\begin{enumerate}[label=(\roman*)]
\item $c_1,\dots,  c_s$ are linearly independent over $\mathbb{K}$. 
\item For each $1 \le i \le m$,  $\sum_{j=1}^s c_j f_{ij} = g_i$.
\end{enumerate}
Since $\sum_{j=1}^s c_j f_{ij}(a) = g_i(a) = 0$ for each $1 \le i \le m$ and $a\in S$,  the $\mathbb{K}$-linear independence of $c_1,  \dots,  c_s$ implies that $f_{ij}(a) = 0$ for each $1 \le i \le m$,  $1 \le j \le s$ and $a\in S$.  Therefore,  $\mathfrak{a}$ is generated by $\{ f_{ij}: 1 \le i \le m,\; 1 \le j \le s \} \subseteq \mathbb{K}[x_{1},\cdots,x_{n}]$.
\end{proof}

\begin{proposition}\label{codim for infinite}
Let $\mathbb{K}$ be an infinite field and let $S$ be a subset of $\mathbb{K}^n$.  If $S \cap  \left( V \setminus \{0\} \right) \ne \emptyset$ for any $r$-dimensional subspace $V\subseteq \mathbb{K}^n$,  then $\dim \overline{S} \ge n-r$ where $\overline{S}$ denotes the Zariski closure of $S$ in $\overline{K}^n$.
\end{proposition} 
\begin{proof}
Denote $R \coloneqq \mathbb{K}[x_1,\dots,  x_n]$ and $\overline{R} \coloneqq \overline{\mathbb{K}}[x_1,\dots,  x_n]$.  By Lemma~\ref{infinite generate},  the ideal $\mathfrak{a} \coloneqq \left\lbrace
f \in \overline{R}: f(a) = 0,\; a \in S
\right\rbrace$ is generated by some $f_{1},\dots,f_{m}\in R$.  In particular,  we have $(\mathfrak{a} \cap R) \otimes_{\mathbb{K}} \overline{\mathbb{K}} = \mathfrak{a}$.  Thus,  we obtain $\height(\mathfrak{a} \cap R)= \height(\mathfrak{a}) \eqqcolon s$ by Lemma~\ref{simple observation}.  Moreover,  Lemma~\ref{lem:dim+ht} implies that $\dim \overline{S} \coloneqq \dim \overline{R}/\mathfrak{a} = n - s$.  Hence it suffices to prove $s \le r$.

According to Proposition~\ref{long regular sequence-2},  there exist $g_{1},\dots,g_{s}\in \spa_{\mathbb{K}}(f_{1},\cdots,f_{m})$ such that $\height(\mathfrak{b})= s$ where $\mathfrak{b} \coloneqq (g_{1},\dots,g_{s})$.   Let $W \subseteq R$ be the vector space consisting of all linear forms and let $M = \{g_1,\dots,  g_s,  l_1,\dots,  l_t\}$ be a maximal subset of $\{g_1,\dots,  g_s\} \cup W$ satisfying: 
\begin{enumerate}[label=(\roman*)]
\item $\{g_1,\dots,  g_s\} \subseteq M$. 
\item For each $0 \le j \le t-1$,  we have $\height(\mathfrak{c}_{j+1}) = \height(\mathfrak{c}_j)+ 1$,  where $\mathfrak{c}_j \coloneqq (g_1,\dots,  g_s,  l_1,\dots , l_j) \subseteq R$.
\end{enumerate}  
Clearly,  elements in $M$ form a regular sequence in $R$ and $t \le n - s$.  By the maximality of $M$,  for any $l \in W$,  either $\height(\mathfrak{c}_t) = \height(\mathfrak{c}_t + (l))$ or $\mathfrak{c}_t + (l)= R$.  

If $s\ge r+1$ but $\mathfrak{c}_t + (l)= R $ for some $l \in W$,  then $n-r\ge n-s+1\ge t+1$.  Let $U \subseteq \mathbb{K}^n$ be the linear subspace defined by $l_1,\dots,  l_t,  l$.  Then $U \cap \overline{S} =\emptyset$, which contradicts to the assumption since $\dim_{\mathbb{K}} U \ge n - (t+1) \ge r$.  

Suppose that $\height(\mathfrak{c}_t) = \height(\mathfrak{c}_t + (l))$ for any $l \in W$.  Lemma~\ref{simple observation} implies that $\min ( \mathfrak{c}_t + (l) ) \subseteq \min ( \mathfrak{c}_t )$.  Hence there is some $\mathfrak{p}\in \min (\mathfrak{c}_t)$ such that $l \in \mathfrak{p}$.  As a consequence,  we have $W = \cup_{\mathfrak{p} \in \min (\mathfrak{c}_t)} (W \cap \mathfrak{p})$.  By Lemma~\ref{uncover},  there is $\mathfrak{q} \in \min(\mathfrak{c}_t)$ such that $(x_{1},\dots,x_{n})\subseteq \mathfrak{q}$,  from which we conclude that $\height(\mathfrak{c}_t) = n$.  In particular,  $t= n -s$ and $X \coloneqq \{v \in \mathbb{K}^n:  f(v) = 0,\; f \in \mathfrak{c}_t\}$ is a finite set.  If there is some $l\in W$ such that $l(v) \ne 0$ for any $v \in X$,  then $l_1,\dots,  l_{n-s},  l$ defines a linear subspace $V$ of $\mathbb{K}^n$ such that $\dim V \ge s-1$ and $S \cap (V \setminus \{0\}) = \emptyset$.  By assumption,  we must have $s-1 \le \dim V \le r - 1$.  

We prove the existence of $l$ by contradiction.  Suppose that every $l \in W$ vanishes at some $v \in X\setminus \{0\}$.  Then $W \subseteq \cup_{v \in X\setminus \{0\}} \mathfrak{a}_v$ where $\mathfrak{a}_v \coloneqq \{f \in R: f(v) = 0\}$.  By Lemma~\ref{uncover} again,  $W \subseteq \mathfrak{a}_v$ for some $v \in X$.  Since $W$ generates the ideal $(x_1,\dots,  x_n)$ in $R$,  we obtain $v = 0$ which contradicts to the choice of $v$.  
\end{proof}    
\section{Properties of the Geometric rank}\label{sec:geo}
In this section,  we explore some properties of the geometric rank.  We begin with the following lemma,  whose proof is omitted since it is verbatim the same as that of Lemma~5.3 in \cite{kopparty2020geometric}.
\begin{lemma}[Lower semi-continuity of $\GR(T)$]\label{Zariski closed of GR}
Let $\mathbb{K}$ be an algebraically closed field.  For any positive integer $r$,  the set $\{T\in \mathbb{K}^{n_1}\otimes\cdots\otimes\mathbb{K}^{n_d}: \GR(T)\le r\}$ is Zariski closed.
\end{lemma}

We recall from \eqref{eq:iso} that each tensor $T \in \mathbb{K}^{n_1}\otimes\cdots\otimes\mathbb{K}^{n_d}$ uniquely determines a multilinear function $f_T$ on $(\mathbb{K}^{n_1})^\ast \times \cdots \times (\mathbb{K}^{n_d})^\ast$.  In particular,  $f_T$ is an element in the ring of polynomials on $(\mathbb{K}^{n_1})^\ast \oplus \cdots \oplus (\mathbb{K}^{n_d})^\ast$. 
\begin{proposition}[Alternative description of $\GR(T)$]\label{computing GR}
For each $T \in \mathbb{K}^{n_1} \otimes \cdots \otimes \mathbb{K}^{n_d}$,  we have 
\[
\GR(T) = \height(\mathfrak{a})  =\min \left\lbrace \codim X_c +c:  c\in \mathbb{N} \right\rbrace,
\]
where $\mathfrak{a}$ is the ideal generated by $\{ f_{T_u}: u \in (\mathbb{K}^{n_d})^\ast \}$,  $X_{c} \coloneqq \{u \in (\mathbb{K}^{n_d})^\ast: \GR(T_u) = c \}$ and $T_u \coloneqq \langle T,  u \rangle \in \mathbb{K}^{n_1} \otimes \cdots \otimes \mathbb{K}^{n_{d-1}}$ for each $u \in (\mathbb{K}^{n_d})^\ast$.
\end{proposition}
\begin{proof}
The first equality follows immediately from Lemma~\ref{lem:dim+ht},  so it suffices to prove the second.  Without loss of generality,  we may assume that $\mathbb{K}$ is algebraically closed.  Denote $Z \coloneqq \{(u_{2},\dots,u_{d})\in (\mathbb{K}^{n_2})^{\ast} \times \cdots \times (\mathbb{K}^{n_d})^{\ast}: \langle T,  u_{2} \otimes \cdots \otimes u_{d} \rangle = 0 \}$.  By definition,   we have 
\[
\langle T,  u_2 \otimes \cdots \otimes u_d \rangle = \langle T_{u_d},  u_2 \otimes \cdots \otimes u_{d-1} \rangle,   \quad \GR(T) = \codim Z.   
\] 
We consider maps 
\begin{align*}
&\tau: (\mathbb{K}^{n_{d}})^{\ast} \to \mathbb{K}^{n_1} \otimes \cdots \otimes \mathbb{K}^{n_{d-1}}, \quad \tau(u) = T_u,  \\
&\pi: Z\to(\mathbb{K}^{n_{d}})^{\ast},\quad \pi(u_{2},\dots,u_{d}) = u_{d}.
\end{align*}
Given an integer $c \ge 0$,  we let $W_c \coloneqq \{u\in (\mathbb{K}^{n_{d}})^{\ast}: \GR(T_{u})\le c\}$.  Then clearly we have 
\[
W_c = \tau^{-1} \left( \{ S \in \mathbb{K}^{n_1} \otimes \cdots \otimes \mathbb{K}^{n_{d-1}}: \GR(S) \le c \} \right).
\]
This together with Lemma~\ref{Zariski closed of GR} implies $W_c$ is Zariski closed.  Moreover,  we have $X_c =W_c \setminus W_{c-1}$ where $W_{-1} \coloneqq \emptyset$.  Then $X_c$ is quasi-affine and $\cup_{c=0}^{m} \pi^{-1}(X_c) = Z$,  where $m \coloneqq n_2 + \cdots + n_{d-1}$.  In fact,  for any $u \in (\mathbb{K}^{n_d})^{\ast}$,  we have $\GR(T_u) \le m$ by definition.  

For each $u \in X_c $,  we notice that $\GR(T_{u}) = c$ and $\pi^{-1}(u) \simeq \{(u_{2},\dots,u_{d-1})\in (\mathbb{K}^{n_2})^{*} \times \cdots \times (\mathbb{K}^{n_{d-1}})^{*}: \langle T_{u}, x_{2} \otimes \cdots \otimes x_{d-1} \rangle = 0\}$.  This implies that $\dim \pi^{-1}(u) = n_2 + \cdots + n_{d-1} - c$ and 
\[
\dim \pi^{-1} (X_c) = \dim X_c + \left( n_2 + \cdots + n_{d-1} - c \right),
\]
by \cite[Proposition~10.6.1~(iii)]{grothendieck1966elements}.  Consequently,  we obtain 
\[
\codim Z =(n_2 + \cdots + n_d) -  \max_{0 \le c \le m } \{ \dim \pi^{-1} (X_c) \} = n_d - \max_{0 \le c \le m} \{ \dim X_c - c \} = \min_{0 \le r \le m} \{\codim X_c + c \}.
\]
For $c > m$,  we have $X_c = \emptyset$.  Thus,  $\codim X_c + c > n_d + m > \codim Z$ and this completes the proof.  
\end{proof}
\begin{remark}
When $d = 3$,  Proposition~\ref{computing GR} coincides with \cite[Theorem~3.1]{kopparty2020geometric}.  For $d\ge 4$,  however, Proposition~\ref{computing GR} is different from the generalization of \cite[Theorem~3.1]{kopparty2020geometric} stated on page~12 of  \cite{kopparty2020geometric},  although its proof is similar to that of \cite[Theorem~3.1]{kopparty2020geometric}.
\end{remark}

By Proposition~\ref{computing GR},  $\codim X_c + c$ provides an upper bound for $\GR(T)$ for any $c \in \mathbb{N}$.  When $\mathbb{K}$ is a finite field,  it is also possible to estimate $\GR(T)$ by $|X_c|$.
\begin{proposition}[Estimate of $\GR(T)$]\label{codim for finite}
Let $m,  c$ be positive integers.  If $T \in \mathbb{F}_q^{n_1} \otimes \cdots \otimes \mathbb{F}_q^{n_d}$ satisfies $|X_c| \ge m$,  then $\GR(T) \le C_2(d) \left( c/C_1(d) + n_d -  \log_q m \right)$ where $C_1$ and $C_2$ are functions on $\mathbb{N}$ in Theorem~\ref{AR stability}.
\end{proposition}
\begin{proof}
We denote $T_u \coloneqq \langle T,  u \rangle$ for $u \in (\mathbb{F}_q^{n_d})^\ast$.  Since $|X_c| \ge m$,  Theorem~\ref{AR stability} implies 
\[
\left\lvert \left\lbrace 
u \in (\mathbb{F}_{q}^{n_d})^\ast:  
\bias(T_u) \ge q^{-c/C_1(d)} 
\right\rbrace
\right\rvert  = 
\left\lvert \left\lbrace 
u \in (\mathbb{F}_{q}^{n_d})^\ast:  \AR( T_u ) \le c/C_1(d)
\right\rbrace
\right\rvert  \ge m.
\]
We have 
\[
q^{-\AR(T)} = \bias(T) = q^{-n_{d}}\sum_{u \in (\mathbb{F}_{q}^{n_d})^\ast} \bias( T_u ) \ge m q^{-c/C_1(d) - n_d},
\]
where the second equality follows from \eqref{eq:bias}.  Hence we conclude by Theorem~\ref{AR stability} that 
\[
\GR(T) \le C_2(d) \AR(T) \le C_2(d) \left(  c/C_1(d)  + n_d - \log_q m \right).  \qedhere
\]
\end{proof}

By the canonical isomorphism \eqref{eq:iso},  each $T\in \mathbb{K}^{n_1} \otimes \cdots \otimes \mathbb{K}^{n_d}$ can be identified with a multilinear polynomial $f_T$ on $(\mathbb{K}^{n_1})^\ast \oplus \cdots \oplus (\mathbb{K}^{n_d})^\ast$.  Hence,  both $\GR(T)$ and $\str(T) \coloneqq \str(f_T)$ are well-defined.  In fact,  they are related by an inequality.  
\begin{proposition}[$\GR(T)$ vs.  $\str(T)$]\label{str vs GR} 
Let $\mathbb{K}$ be a field.  For each $T\in \mathbb{K}^{n_1} \otimes \cdots \otimes \mathbb{K}^{n_d}$,  we have $2\str(T) \ge \GR(T)$.
\end{proposition} 
\begin{proof}
We let $\{x_{j}:  1 \le j \le n_1\}$ be a basis of $\mathbb{K}^{n_1}$.  Then we observe 
\begin{align*}
\GR(T) &= \codim \left( \{ (u_2,\dots,  u_d) \in (\mathbb{K}^{n_2})^\ast \times \cdots \times (\mathbb{K}^{n_d})^\ast:  \langle T,  u_2 \otimes \cdots \otimes u_d  \rangle = 0 \} \right) \\
&= \codim \left( \{(u_2,\dots,  u_d) \in (\mathbb{K}^{n_2})^\ast \times \cdots \times (\mathbb{K}^{n_d})^\ast:  \partial f_T /\partial x_j (u_2,\dots,  u_d) = 0,\; 1 \le j \le n_1 \} \right) \\
&= \height (  \mathfrak{a} ), 
\end{align*}
where $\mathfrak{a} \coloneqq ( \partial f_T/\partial x_1,\dots,   \partial f_T/\partial x_{n_1})$ and the last equality follows from Lemma~\ref{lem:dim+ht}.  By Lemma~\ref{lem:str-height},  we have $2 \str(T)\ge \height(\mathfrak{a}) = \GR(T)$.
\end{proof}

Let $V$ be a $\mathbb{K}$-linear subspace of $\mathbb{K}^{n_1} \otimes \cdots \otimes \mathbb{K}^{n_d}$.  We define \[
\GR(V) \coloneqq \min \{\GR(T):  T \in V\},\quad \PR(V) \coloneqq \min \{\PR(T):  T\in V\}.
\]
We denote $V^{\overline{\mathbb{K}}}\coloneqq V \otimes_{\mathbb{K}} \overline{\mathbb{K}}$.  Then $V^{\overline{\mathbb{K}}}$ is a $\overline{\mathbb{K}}$-linear subspace of $\overline{\mathbb{K}}^{n_1} \otimes_{\overline{\mathbb{K}}} \cdots \otimes_{\overline{\mathbb{K}}} \overline{\mathbb{K}}^{n_d}$ and by definition,  
\[
\GR ( V^{\overline{\mathbb{K}}} ) \coloneqq \left\lbrace \GR_{\overline{\mathbb{K}}} (T):  T \in V^{\overline{\mathbb{K}}} \right\rbrace,\quad \PR ( V^{\overline{\mathbb{K}}} ) \coloneqq \left\lbrace \PR_{\overline{\mathbb{K}}} (T):  T \in V^{\overline{\mathbb{K}}} \right\rbrace.
\]
Here $\GR_{\overline{\mathbb{K}}} (T)$ (resp.   $\PR_{\overline{\mathbb{K}}} (T)$) is the geometric rank (resp.  partition rank) of $T \in V^{\overline{\mathbb{K}}}$ as a tensor over $\overline{\mathbb{K}}$.
\begin{proposition}[Stability of $\GR(V)$]\label{minGR stability}
There exists a function $J$ on $\mathbb{N} \times \mathbb{N} \times \mathbb{N}$ such that for any linear subspace $V\subseteq \mathbb{K}^{n_{1}}\otimes\cdots\otimes\mathbb{K}^{n_{d}}$ with $\dim V \le m$ and $\GR(V^{\overline{\mathbb{K}}}) \le r$,  we have $\GR(V) \le J(d,  m, r)$.
\end{proposition}
\begin{proof}  
We recall from \cite[Theorem~5]{kopparty2020geometric} that $\GR(T) \le \PR(T)$ for any $T\in \mathbb{K}^{n_{1}}\otimes\cdots\otimes\mathbb{K}^{n_{d}}$.  This implies that $\GR(V)\le \PR(V)$.  By \cite[Theorem~1.7.2]{bik2025strength},  there is a function $C$ on $\mathbb{N}$ such that $\PR(V) \le  C(d) m^3 s^{d-1} \log\left( s + m \right)$,  where $s \coloneqq \PR(V^{\overline{\mathbb{K}}})$.  According to \cite[Corollary~3]{cohen2023partition},  we also have $\PR_{\overline{K}} (T) \le (2^{d-1} - 1) \GR_{\overline{K}}(T)$ for each $T \in V^{\overline{\mathbb{K}}}$.  Thus,  $s \le (2^{d-1} - 1) r$ and  
\[
\GR(V) \le \PR(V) \le C(d) m^3 s^{d-1} \log(s+ m)\le J(d,m,r), 
\]
where $J(d,m,r) \coloneqq  C(d) m^3 (2^{d-1} - 1)^{d-1} r^{d-1} \log( (2^{d-1} - 1)r  + m)$.
\end{proof}
\section{Proof of Theorem~\ref{thm:general}}\label{sec:1.1}
The goal of this section is to prove Theorem~\ref{thm:general}.  We first establish the lemma that follows.    
\begin{lemma}\label{d-minGR-2} 
There are functions $M_1$ and $M_2$ on $\mathbb{N} \times \mathbb{N}$ with the following property.  Suppose $\mathbb{K}$ is a field and $T\in \mathbb{K}^{n_{1}}\otimes\cdots\otimes\mathbb{K}^{n_{d}}$ is a tensor with $d$-slices $T_1,\dots,  T_c \in \mathbb{K}^{n_{1}}\otimes\cdots\otimes\mathbb{K}^{n_{d-1}}$ satisfying 
\begin{enumerate}[label = (\roman*)] 
\item Either $\ch(\mathbb{K}) = 0$ or $\ch(\mathbb{K}) > d$.
\item $T_1,\dots, T_c$ are $\mathbb{K}$-linearly independent. 
\item $\GR \left(  \spa_{\mathbb{K}} \{T_1,\dots,  T_c \} \right) \ge M_1(d,c)$. 
\end{enumerate}
Then there is a degree $M_2(d,c)$ extension $\mathbb{L}/\mathbb{K}$ such that $\Q_{\mathbb{L}} (T) \ge c$.  Here $\Q_{\mathbb{L}} (T)$ denotes the subrank of $T$ as a tensor over $\mathbb{L}$ via the natural inclusion $\mathbb{K}^{n_1} \otimes \cdots \otimes \mathbb{K}^{n_d} \subseteq \mathbb{L}^{n_1} \otimes \cdots \otimes \mathbb{L}^{n_d}$.
\end{lemma}
\begin{proof}
Let $I$ be the function in Proposition~\ref{small subalgebra general field} and let $M_1: \mathbb{N} \times \mathbb{N}$ be a function such that $M_1(d,c) > 2 I(d,  c^d - c(c-1)^{d-1} + (d-1)c)$ for any $(d,c) \in \mathbb{N} \times \mathbb{N}$.  Denote $\mathbb{L}_0 \coloneqq \mathbb{K}$.  For each $1 \le s \le c$,  we claim that there exist a function $C$ on $\mathbb{N} \times \mathbb{N} \times \mathbb{N}$,  an extension $\mathbb{L}_s$ of $\mathbb{L}_{s-1}$,  vectors $u^{(i)}_{j} \in (\mathbb{L}_s^{n_i})^\ast$ where $1 \le i \le d-1$,  $1 \le j \le s$,  and tensors $ T'_{k} \in \spa_{\mathbb{L}_s} \lbrace T_1,\dots,  T_c \rbrace$ where $1 \le k \le c$ such that  
    \begin{enumerate} [label=(\arabic*)]
    \item $[\mathbb{L}_s:\mathbb{L}_{s-1}]\le C(s,  d,  c)$.  \label{d-minGR-2:item1}
        \item For each $1 \le i \le d-1$, $u^{(i)}_{1},\cdots,u^{(i)}_{s}$ are linearly independent over $\mathbb{L}_s$.  \label{d-minGR-2:item2}
         \item $T'_{1},\dots,T'_{c}$ are linearly independent over $\mathbb{L}_s$.  \label{d-minGR-2:item3}
        \item $\left\langle T'_{k},  u^{(1)}_{j_1} \otimes \cdots \otimes u^{(d-1)}_{j_{d-1}} \right\rangle =\delta(k,  j_1,\cdots,j_{d-1})$ for each $1 \le k \le c$ and $1 \le j_1,\dots,  j_{d-1}\le s$, where $\delta$ is the Kronecker symbol.  \label{d-minGR-2:item4}
    \end{enumerate}
If the claim is true for $s=c$,  then we let $\mathbb{L} \coloneqq \mathbb{L}_c$ so that $\Q_{\mathbb{L}}(T) \ge c$ and $[\mathbb{L} : \mathbb{K}] \le M_2(d,c) \coloneqq \prod_{s=1}^c C(s,d,c)$.

It is left to prove the claim.  We proceed by induction on $s$.  For $s = 1$,  we observe that the polynomial $f_{T_1}$ corresponding to $T_1$ via the isomorphism \eqref{eq:iso} is nonzero and multilinear.  By \cite[Lemma~2.1]{alon1999combinatorial},  there exists some $(u^{(1)}_1,\dots,u^{(d-1)}_1) \in (\mathbb{K}^{n_1})^\ast \oplus \cdots \oplus (\mathbb{K}^{n_{d-1}})^\ast$ such that  
\[
\left\langle T_{1},  u^{(1)}_1\otimes \cdots \otimes u^{(d-1)}_1 \right\rangle  = 1.
\]
We let $\mathbb{L}_1 = \mathbb{K} $ and set
\[
T'_1 \coloneqq T_1,\quad T'_{k} \coloneqq T_{k}- \left\langle T_{k},  u^{(1)}_1\otimes \cdots \otimes u^{(d-1)}_1 \right\rangle T_1,\quad 2 \le k \le c.
\] 
It is straightforward to verify that $\mathbb{L}_1$,  $\{u^{(1)}_1,\dots,  u^{(d-1)}_1\}$ and $\{ T'_1,\dots,  T'_c \}$ satisfy \ref{d-minGR-2:item1}--\ref{d-minGR-2:item4}.

Assume that the claim is true for each $s \le m < c$.  Then there is a tower of field extensions $\mathbb{L}_m/\mathbb{L}_{m-1} / \cdots /\mathbb{L}_1/\mathbb{K}$ such that $[\mathbb{L}_{s} : \mathbb{L}_{s-1}] \le C(s,  c,d)$ for any $1 \le s \le m$.  Here $C$ is a function on $\mathbb{N} \times \mathbb{N} \times \mathbb{N}$.  We let $u^{(i)}_{j} \in (\mathbb{L}_m^{n_{i}})^\ast$ where $1 \le i \le d-1$,  $1 \le j \le m$ and $T'_{k} \in \spa_{\mathbb{L}_m} \left\lbrace T_1,\dots,  T_c \right\rbrace$ where $1 \le k \le c$ be vectors and tensors satisfying \ref{d-minGR-2:item2}--\ref{d-minGR-2:item4}.  We prove the claim for $s = m + 1$ by constructing the desired $u$'s and $T'$'s from the ones for $s = m$,  in the following three steps: 
\begin{enumerate}[label= Step \arabic*.]
\item \label{d-minGR-2:step1} For each $1 \le i \le d-1$,  we denote $V_i \coloneqq \spa_{\mathbb{L}_m} \left\lbrace u^{(i)}_1,  \dots,u^{(i)}_m \right\rbrace \subseteq (\mathbb{L}_m^{n_i})^\ast$.  We first show that there is an extension $\mathbb{L}_{m+1}/\mathbb{L}_{m}$ and $v^{(i)}\in (\mathbb{L}_{m+1}^{n_{i}})^\ast \setminus V_i^{\mathbb{L}_{m+1}}$ where $1\le i\le d-1$,  such that 
\begin{enumerate}[label=(\alph*)]
\item $[\mathbb{L}_{m+1}:\mathbb{L}_{m}]\le C(m+1,  c,d)$ for some function $C$ on $\mathbb{N} \times \mathbb{N} \times \mathbb{N}$.\label{d-minGR-2:itema}
\item $\left\langle T'_{m+1},  v^{(1)}\otimes \cdots \otimes v^{(d-1)} \right\rangle = 1$.  \label{d-minGR-2:itemb}
\item For each $(d-1)$-tuple $\left( w^{(1)},  \dots,  w^{(d-1)} \right) \in \prod_{j=1}^{d-1} \left\lbrace u^{(i)}_{1},\dots,u^{(i)}_m,  v^{(i)}\right\rbrace$,  we have 
\[
\left\langle T'_k,  w^{(1)}\otimes \cdots \otimes w^{(d-1)} \right\rangle =  0,\quad 1\le k \le c,
\] 
if $w^{(i)}\neq v^{(i)}$ for some $1 \le i \le d-1$.\label{d-minGR-2:itemc}
\end{enumerate}
We denote by $f_T$ the multilinear polynomial on $(\mathbb{L}_s^{n_1})^\ast \oplus \cdots \oplus (\mathbb{L}_s^{n_{d-1}})^\ast$ determined by a tensor $T\in \mathbb{L}_s^{n_1} \otimes \cdots \otimes \mathbb{L}_s^{n_{d-1}}$ via the isomorphism \eqref{eq:iso}.  Then \ref{d-minGR-2:itemb} and \ref{d-minGR-2:itemc} are respectively equivalent to 
\begin{enumerate}[label=(\alph*'),start = 2]
\item $f_{T'_{m+1}} (v^{(1)},  \dots,  v^{(d-1)}) \ne 0$. \label{d-minGR-2:itemb'} 
\item If $w^{i} = v^{(i)}$ for some $1 \le i \le d-1$,  then $f_{T'_k}(w^{(1)},  \dots,  w^{(d-1)}) = 0$ for all $1 \le k \le c$.\label{d-minGR-2:itemc'}
\end{enumerate}
Let $R$ be the polynomial ring on $(\mathbb{L}_m^{n_1})^\ast \oplus \cdots \oplus (\mathbb{L}_m^{n_{d-1}})^\ast$,  and we regard $v^{(1)}, \dots,  v^{(d-1)}$ in \ref{d-minGR-2:itemb'} and \ref{d-minGR-2:itemc'} as vectors of independent indeterminates. We observe that there is only one form $f_{T'_{m+1}} \in R$ in \ref{d-minGR-2:itemb'} and $\deg f_{T'_{m+1}} \le d - 1$.  Since $u^{(i)}_j$,  $1 \le i \le d-1$,  $1 \le j \le m$ are fixed vectors,  there are $N \coloneqq c\sum_{i=1}^{d-2}\binom{d-1}{i}m^{d-1-i}$ forms in \ref{d-minGR-2:itemc'},  each of which has degree at most $d - 2$.  Since $m  \le  c - 1$,  we may obtain 
\[
 N  \le c\sum_{i=1}^{d-2}\binom{d-1}{i} (c-1)^{d-1-i} \le  c(c^{d-1} - (c-1)^{d-1} ). 
 \]
To simplify notations,  we denote the form in \ref{d-minGR-2:itemb'} by $g$,  and the $N$ forms in \ref{d-minGR-2:itemc'} by $f_1,\dots,  f_N$.  Let $n \coloneqq n_1 + \cdots + n_{d-1}$.  We choose a basis $\alpha_1,\dots,  \alpha_n$ of $\mathbb{L}_{m}^{n_{1}}\oplus\cdots\oplus\mathbb{L}_{m}^{n_{d-1}}$ such that the restriction of $\alpha_1,\dots,  \alpha_{(d-1)m}$ to $V_1 \oplus \cdots \oplus V_{d-1}$ is a basis of $(V_1 \oplus \cdots \oplus V_{d-1})^\ast$.  
\begin{enumerate}[label = Step 1.\arabic*.]
\item \label{d-minGR-2:step1.1} We apply Proposition~\ref{small subalgebra general field} to obtain an $R_1$-sequence of forms $g_1,\dots,  g_t \in R$ such that $t \le I(d,  N + (d-1)m)$,  $\left\lbrace f_1,\dots,  f_N,  \alpha_1,\dots,  \alpha_{(d-1)m}\right\rbrace \subseteq \mathbb{L}_m[g_1,\dots,  g_t]$ and $\deg(g_i) \le \max\{d-1,1\}$ for each $1 \le i \le t$.  By the choice  $M_1(d,c)$,  we have $M_1(d,c) > 2 I(d,  N + (d-1)m)$.  This implies $g \not\in (g_1,\dots,  g_t)$.  Otherwise,  we have $t \ge \str(g)$.  Then Proposition~\ref{minGR stability} and Proposition~\ref{str vs GR} would lead to a contradiction:
\begin{align*}
I(d,  N + (d-1)m)   \ge \GR\left( \spa_{\mathbb{K}} \{ T_1,\dots ,  T_c\} \right)/2 \ge M_1(d,c)/2.
\end{align*}
\item \label{d-minGR-2:step1.2} According to Proposition~\ref{solution of constructable set},  there is an extension $\mathbb{L}_{m+1}/\mathbb{L}_m$ of degree at most $C(m+1,d,t) \coloneqq L(d,t)$ such that 
\[
f_1(v) = \cdots = f_t(v) = \alpha_1(v) = \cdots = \alpha_{(d-1)m}(v) = 0,\quad g(v) = 1
\]
for some nonzero $v \coloneqq (v^{(1)},  \dots,  v^{(d-1)}) \in (\mathbb{L}_{m+1}^{n_1})^\ast \oplus \cdots \oplus (\mathbb{L}_{m+1}^{n_{d-1}})^\ast$.  Here $L$ is the function in Proposition~\ref{solution of constructable set}.  Since $v \ne 0$ and $\{\alpha_1,\dots,  \alpha_{(d-1)m}\}$ restricts to a basis of $(V_1 \oplus \cdots \oplus V_{d-1})^\ast$,  we must have $v \not\in V_1^{\mathbb{L}_{m+1}} \oplus \cdots \oplus V_{d-1}^{\mathbb{L}_{m+1}}$.  Therefore,  we find $\mathbb{L}_{m+1}/\mathbb{L}_m$ and $(v^{(1)}, \dots,  v^{(d-1)}) \in \left( (\mathbb{L}_{m+1}^{n_1})^\ast \setminus V_1^{\mathbb{L}_{m+1}} \right) \times \cdots \times \left( (\mathbb{L}_{m+1}^{n_{d-1}})^\ast \setminus V_{d-1}^{\mathbb{L}_{m+1}} \right)$ satisfying \ref{d-minGR-2:itema}--\ref{d-minGR-2:itemc}. 
\end{enumerate}
\item \label{d-minGR-2:step2} Now we are ready to construct $u$'s and $T'$'s for $s = m +1$.  For each $1 \le i \le d-1$,  we let $u^{(i)}_{m+1} \coloneqq v^{(i)}$ and 
\[
\widehat{T}'_k \coloneqq \begin{cases}
T'_{m+1},  &\text{if k = m+1} \\
T'_k - \left\langle T'_k,  u^{(1)}_{m+1} \otimes \cdots u^{(d-1)}_{m+1} \right\rangle T'_{m+1},  &\text{if~} 1 \le k  \le c  \text{~and~} k\neq m+1
\end{cases}
\]
\item  \label{d-minGR-2:step3} We verify that $u^{(i)}_j$ and $\widehat{T}'_k$ satisfy \ref{d-minGR-2:item2}--\ref{d-minGR-2:item4},  where $1 \le i \le d-1$,  $1 \le j \le m+1$ and $1 \le k \le c$.  For $1 \le k \le c$ and $k\neq m+1$,  we have 
\begin{align*}
\left\langle \widehat{T}'_k,  u^{(1)}_{j_{1}}\otimes \cdots \otimes u^{(d-1)}_{j_{d-1}} \right\rangle
 &=  \left\langle T'_k,  u^{(1)}_{j_1}\otimes \cdots \otimes u^{(d-1)}_{j_{d-1}}\right\rangle   \\
&- \left\langle T'_k,  u^{(1)}_{m+1}\otimes \cdots \otimes u^{(d-1)}_{m+1}\right\rangle
\left\langle T'_{m+1},  u^{(1)}_{j_{1}}\otimes \cdots \otimes u^{(d-1)}_{j_{d-1}} \right\rangle.
\end{align*}
If $(j_{1},\dots,j_{d-1})\neq(m+1,\dots,m+1)$,  then by \ref{d-minGR-2:itemc} and the induction hypothesis,  we may derive $\left\langle T'_{m+1},  u^{(1)}_{j_{1}}\otimes \cdots \otimes u^{(d-1)}_{j_{d-1}} \right\rangle=0$ and 
\[
\left\langle \widehat{T}'_k,  u^{(1)}_{j_{1}}\otimes \cdots \otimes u^{(d-1)}_{j_{d-1}} \right\rangle = \left\langle T'_k,  u^{(1)}_{j_1}\otimes \cdots \otimes u^{(d-1)}_{j_{d-1}}\right\rangle  = \delta (k,j_1,\dots,  j_{d-1}).
\] 
If $(j_{1},\dots,j_{d-1})=(m+1,\cdots,m+1)$,  then  \ref{d-minGR-2:itemb} implies 
\begin{align*}
\left\langle \widehat{T}'_k,  u^{(1)}_{m+1}\otimes \cdots \otimes u^{(d-1)}_{m+1} \right\rangle
 &=  \left\langle T'_k,  u^{(1)}_{m+1}\otimes \cdots \otimes u^{(d-1)}_{m+1}\right\rangle   \\
&- \left\langle T'_k,  u^{(1)}_{m+1}\otimes \cdots \otimes u^{(d-1)}_{m+1}\right\rangle
\left\langle T'_{m+1},  u^{(1)}_{m+1}\otimes \cdots \otimes u^{(d-1)}_{m+1} \right\rangle \\
&=0. 
\end{align*}
For $k = m+1$,   we have 
\[
\left\langle \widehat{T}'_k,  u^{(1)}_{j_1}\otimes \cdots \otimes u^{(d-1)}_{j_{d-1}} \right\rangle = 
\left\langle T'_{m+1},  u^{(1)}_{j_1}\otimes \cdots \otimes u^{(d-1)}_{j_{d-1}} \right\rangle.
\]
If $(j_{1},\dots, j_{d-1})\neq(m+1,\dots,m+1)$,  then we may derive $\left\langle T'_{m+1},  u^{(1)}_{j_1}\otimes \cdots \otimes u^{(d-1)}_{j_{d-1}} \right\rangle = 0$ from \ref{d-minGR-2:itemc} and the induction hypothesis.  If $(j_{1},\dots, j_{d-1}) = (m+1,\dots,m+1)$,  then \ref{d-minGR-2:itemb} leads to $\left\langle T'_{m+1},  u^{(1)}_{m+1}\otimes \cdots \otimes u^{(d-1)}_{m+1} \right\rangle = 1$.  \qedhere
\end{enumerate}
 \end{proof}
    
\begin{proof}[Proof of Theorem~\ref{thm:general}]
Let $M_1$ and $M_2$ be the functions on $\mathbb{N} \times \mathbb{N}$ in Lemma~\ref{d-minGR-2}.  We set $B \coloneqq M_2$.  Recall from Proposition~\ref{computing GR} that
\[
\GR (T) = \min \{ \codim X_r + r : r \in \mathbb{N} \},  
\]
where $X_r = \{ u \in (\mathbb{K}^n_d)^\ast: \GR( \left\langle T,  u \right\rangle ) = r \}$ for each $r \in \mathbb{N}$.  Let $s$ be the smallest integer such that $\Q_{\mathbb{F}} (T) \le s-1$ for any degree $M_2(d,s-1)$ extension $\mathbb{F}/\mathbb{K}$.  Then by Lemma~\ref{d-minGR-2},  for each $s$-dimensional subspace $V$ of $(\mathbb{K}^{n_d})^\ast$,  there exists some $v \in V \setminus \{0\}$ such that $\GR \left( \left\langle T, v \right\rangle  \right) \le M_1(d,s)-1$.  We notice that by the minimality of $s$,  there exists some $\mathbb{L}/\mathbb{K}$ of degree $M_2(d,s-1)$ such that $\Q_{\mathbb{L}} (T) = s-1$.  Since $\Q_{\mathbb{L}} (T) \le \GR(T)$ by \cite[Theorem~5]{kopparty2020geometric},  we obtain 
\begin{align*}
[\mathbb{L}: \mathbb{K}] &= M_2(d,s-1) \le B(d,  \GR(T)).
\end{align*} 
We define
\[
A(d,  s) \coloneqq \max \left\lbrace
s + M_1(d,  s + 1) ,   C_2(d) \left[  \frac{M_1(d,s)-1}{C_1(d)} 
+ s + 2      \right]
\right\rbrace. 
\]
Here $C_1$ and $C_2$ are the functions in Theorem~\ref{AR stability}.  

The rest of the proof is split according to whether $\mathbb{K}$ is a finite field.  We first suppose that $\mathbb{K}$ is an infinite field.  According to Proposition~\ref{codim for infinite},  we have $\codim X_{M_1(d,s)-1} \le s$ which implies $\GR(T) \le s + M_1(d,s) - 1$.  Thus,  we have $\GR(T) \le A (d,  \Q_{\mathbb{L}}(T))$.    

Next we assume $\mathbb{K} = \mathbb{F}_q$ and denote $m \coloneqq |X_{M_1(d,s)-1}|$.  For each $v \in X_{M_1(d,s)-1} \setminus \{0\}$,  we define 
\[
\sigma(v) \coloneqq \left\lbrace
V \in \Gr(s,  n_d):  v \in V
\right\rbrace
\]
and notice that $\sigma(v) \simeq \Gr(s-1,  n_d-1)$. Since $\cup_{v \in X_{M_1(d,s)-1}\setminus \{0\}} \sigma(v) = \Gr(s,n_d)$,  we obtain
\[
m \ge \frac{\left\lvert \Gr(s, n_d) \right\rvert}{|\Gr(s-1,n_d-1)|} = \frac{q^{n_{d}}-1}{q^{s}-1}\ge q^{n_{d}-s-2}, 
\]
where the equality follows from Lemma~\ref{counting subspace}.  By Proposition~\ref{codim for finite},  there are functions $C_1$ and $C_2$ such that 
\[
\GR(T) \le C_2(d) \left[  \frac{M_1(d,s)-1}{C_1(d)} 
+ s + 2      
\right]. \qedhere
\]
\end{proof}

\section{Proof of Theorem~\ref{thm:general1}}\label{sec:thm1.2}
The following lemma is a parallel version of Lemma~\ref{d-minGR-2} for algebraically closed field. 
\begin{lemma}\label{lem:closed} 
There is a function $M$ on $\mathbb{N} \times \mathbb{N}$ with the following property.  Suppose $\mathbb{K}$ is a field and $T\in \mathbb{K}^{n_{1}}\otimes\cdots\otimes\mathbb{K}^{n_{d}}$ is a tensor with $d$-slices $T_1,\dots,  T_c \in \mathbb{K}^{n_{1}}\otimes\cdots\otimes\mathbb{K}^{n_{d-1}}$ satisfying 
\begin{enumerate}[label = (\roman*)] 
\item $\mathbb{K}$ is algebraically closed.
\item $T_1,\dots, T_c$ are $\mathbb{K}$-linearly independent. 
\item $\GR \left(  \spa_{\mathbb{K}} \{T_1,\dots,  T_c \} \right) \ge M(d,c)$. 
\end{enumerate}
Then we have $\Q (T) \ge c$.  
\end{lemma}
\begin{proof}
Let $D$ be the function in Theorem~\ref{algebraically closed small subalgebra} and let $M: \mathbb{N} \times \mathbb{N}$ be a function such that $M(d,c) > 2 D(d,  c^d - c(c-1)^{d-1} + (d-1)c)$ for any $(d,c) \in \mathbb{N} \times \mathbb{N}$.  

We claim that for each $1 \le s \le c$,  there exist vectors $u^{(i)}_{j} \in (\mathbb{K}^{n_i})^\ast$ where $1 \le i \le d-1$,  $1 \le j \le s$,  and tensors $ T'_{k} \in \spa_{\mathbb{K}} \lbrace T_1,\dots,  T_c \rbrace$ where $1 \le k \le c$ such that  
    \begin{enumerate} [label=(\arabic*)]
        \item For each $1 \le i \le d-1$, $u^{(i)}_{1},\cdots,u^{(i)}_{s}$ are linearly independent.  \label{lem:closed:item1}
         \item $T'_{1},\dots,T'_{c}$ are linearly independent over $\mathbb{L}_s$.  \label{lem:closed:item2}
        \item $\left\langle T'_{k},  u^{(1)}_{j_1} \otimes \cdots \otimes u^{(d-1)}_{j_{d-1}} \right\rangle =\delta(k,  j_1,\cdots,j_{d-1})$ for each $1 \le k \le c$ and $1 \le j_1,\dots,  j_{d-1}\le s$, where $\delta$ is the Kronecker symbol.  \label{lem:closed:item3}
    \end{enumerate}
If the claim is true,  then it is straightforward to verify that $\Q(T) \ge c$. 

The claim is proved by induction on $s$.  The proof for $s = 1$ is the same as that in Lemma~\ref{d-minGR-2}.  Assume that the claim is true for $s = m < c$.  We let $u^{(i)}_{j} \in (\mathbb{K}^{n_i})^\ast$ where $1 \le i \le d-1$,  $1 \le j \le m$ and $T'_{k} \in \spa_{\mathbb{K}} \left\lbrace T_1,\dots,  T_c \right\rbrace$ where $1 \le k \le c$ be vectors and tensors satisfying \ref{lem:closed:item1}--\ref{lem:closed:item3}.  Next we prove the claim for $s = m + 1$.  For each $1 \le i \le d-1$,  we denote $V_i \coloneqq \spa_\mathbb{K} \left\lbrace u^{(i)}_1,  \dots,u^{(i)}_m \right\rbrace \subseteq (\mathbb{K}^{n_i})^\ast$.  Repeating the argument for \ref{d-minGR-2:step1} in the proof of Lemma~\ref{d-minGR-2},  with Proposition~\ref{small subalgebra general field} replaced by Theorem~\ref{algebraically closed small subalgebra} in \ref{d-minGR-2:step1.1},  we can show that there are $v^{(i)}\in (\mathbb{K}^{n_{i}})^\ast \setminus V_i$ where $1\le i\le d-1$,  such that 
\begin{enumerate}[label=(\alph*)]
\item $\left\langle T'_{m+1},  v^{(1)}\otimes \cdots \otimes v^{(d-1)} \right\rangle = 1$.  \label{lem:closed:itema}
\item For each $(d-1)$-tuple $\left( w^{(1)},  \dots,  w^{(d-1)} \right) \in \prod_{j=1}^{d-1} \left\lbrace u^{(i)}_{1},\dots,u^{(i)}_m,  v^{(i)}\right\rbrace$,  we have 
\[
\left\langle T'_k,  w^{(1)}\otimes \cdots \otimes w^{(d-1)} \right\rangle =  0,\quad 1\le k \le c,
\] 
if $w^{(i)}\neq v^{(i)}$ for some $1 \le i \le d-1$.\label{lem:closed:itemb}
\end{enumerate}
The rest of the proof is verbatim the same as \ref{d-minGR-2:step2} and \ref{d-minGR-2:step3} in the proof of Lemma~\ref{d-minGR-2}.
\end{proof}

\begin{proof}[Proof of Theorem~\ref{thm:general1}]
We denote $c \coloneqq \Q(T)$ and $T_u \coloneqq \left\langle T,  u \right\rangle$ for each $u \in (\mathbb{K}^{n_d})^{\ast}$.  By Lemma~\ref{lem:closed},  for any $(c+1)$-dimensional subspace $V$ of $\mathbb{K}^{n_{d}}$,  there exists $v\in V\setminus \{0\}$ such that $\GR( T_v ) < M
(d,c+1)$.  We consider 
\[
X \coloneqq \{ u \in (\mathbb{K}^{n_3})^{\ast}: \rank (T_u)\le M
(d,c+1) \}.
\] 
Since $\mathbb{K}$ is algebraically closed, Proposition~\ref{codim for infinite} implies that $\codim X \le c+1$.  By Proposition~\ref{computing GR},  we conclude that $\GR(T) \le M(d,c+1) + c + 1 \eqqcolon C(d,  c)$.
\end{proof}

\section{Proof of Theorem~\ref{3-GR<Q infinite}}\label{sec:thm1.4}
The idea that underlies the proof of Theorem~\ref{3-GR<Q infinite} is similar to those of Theorems~\ref{thm:general}. and \ref{thm:general1}.  The only difference is that we may invoke Lemma~\ref{3-minGR} for $d = 3$ instead of Lemma~\ref{d-minGR-2} or Lemma~\ref{lem:closed} for general $d$.
\begin{proof}[Proof of Theorem~\ref{3-GR<Q infinite}]
We denote $c \coloneqq \Q(T)$ and $T_u \coloneqq \left\langle T,  u \right\rangle$ for each $u \in (\mathbb{K}^{n_3})^{\ast}$.  By Lemma~\ref{3-minGR},  for any $(c+1)$-dimensional subspace $V$ of $\mathbb{K}^{n_{3}}$,  there exists $v\in V\setminus \{0\}$ such that $\GR( T_v ) < 2(c+1)c$.  We consider 
\[
X \coloneqq \{ u \in (\mathbb{K}^{n_3})^{\ast}: \rank (T_u)\le 2(c+1)c - 1 \}.
\]

Suppose that $\mathbb{K}$ is an infinite field.  Then Proposition~\ref{codim for infinite} implies that $\codim X \le c+1$.  By Proposition~\ref{computing GR},  we conclude that 
\[
\GR(T) \le 2(c+1)c - 1 + c + 1 = 2c^2 + 3c. 
\] 

Next we assume $\mathbb{K} = \mathbb{F}_q$ and denote $m \coloneqq |X|-1$.  For each $v \in X \setminus \{0\}$,  we define 
\[
\sigma(v) \coloneqq \left\lbrace
V \in \Gr(c+1,  n_3):  v \in V
\right\rbrace
\]
and notice that $\sigma(v) \simeq \Gr(c,  n_3-1)$.  Since $\cup_{v \in X\setminus \{0\}} \sigma(v) = \Gr(c+1,n_3)$,  we obtain
\[
m \ge \frac{\left\lvert \Gr(c+1, n_3) \right\rvert}{|\Gr(c,n_3-1)|} = \frac{q^{n_{3}}-1}{q^{c+1}-1}\ge q^{n_{3}-c-2},
\]
where the equality follows from Lemma~\ref{counting subspace}.  By Proposition~\ref{codim for finite},  there are constants $c_1$ and $c_2$ such that 
\[
\GR(T) \le c_2 \left[  \frac{2(c+1)c - 1}{c_1} 
+ c + 2         
\right] = \frac{2c_2 }{c_1} c^2 + \left( \frac{2c_2 }{c_1} + c_2  \right) c + \left( 2c_2  - \frac{c_2 }{c_1} \right). \qedhere
\]
\end{proof}

\section{Conclusion}
The rate of growth of functions in Theorems~\ref{thm:general} and \ref{thm:general1} with respect to $\Q(T)$ remains elusive.  However,  Theorem~\ref{3-GR<Q infinite} and \cite[Theorem~1.4]{derksen2024subrank} suggest the conjecture that follows is reasonable.
\begin{conjecture}[Polynomial bound of $\PR(T)$ in terms of $\Q(T)$]\label{conj:bound}
Given a positive integer $d$,  there exist a polynomial $\varphi$ of degree at most $(d-1)$ with the following property.  For any field $\mathbb{K}$ and any tensor $T \in \mathbb{K}^{n_1} \otimes \cdots \otimes \mathbb{K}^{n_d}$,  we have $\GR(T) = \varphi(\Q(T))$.  
\end{conjecture}

If Conjecture~\ref{conj:bound} is true,  then the same argument for Corollary~\ref{cor:Qstability} implies the stability of $\Q(T)$ under field extensions.  Consequently,  this leads us to the following conjecture,  which is a priori weaker than Conjecture~\ref{conj:bound}.
\begin{conjecture}[Stability of $\Q(T)$ under field extensions]\label{conj:stability}
Given a positive integer $d$,  there exist a polynomial $\psi$ of degree at most $(d-1)$ with the following property.  For any field $\mathbb{K}$ and any tensor $T \in \mathbb{K}^{n_1} \otimes \cdots \otimes \mathbb{K}^{n_d}$,  we have $\Q_{\overline{\mathbb{K}}}(T) = \psi(\Q(T))$.  
\end{conjecture}

\subsection*{Acknowledgment} We would like to thank Tigran Ananyan and Longteng Chen for helpful discussions.
\bibliographystyle{abbrv}
\bibliography{ref}
\end{document}